\documentclass{amsart}

\usepackage[draft=false,pagebackref]{hyperref}
\usepackage{esint,amssymb,bm,mathtools}

\usepackage{xcolor}
\newtheorem{thm}[equation]{Theorem}
\newtheorem{lem}[equation]{Lemma}

\theoremstyle{definition}
\newtheorem{defn}[equation]{Definition}
\newtheorem{rmk}[equation]{Remark}

\numberwithin{equation}{section}

\newcommand\abs[2][empty]{\csname#1\endcsname \lvert{#2}\csname#1\endcsname\rvert}
\newcommand\doublebar[2][empty]{\csname#1\endcsname \lVert{#2}\csname#1\endcsname\rVert}

\newcommand\mat[1]{\bm{#1}}
\newcommand\arr[1]{\bm{\dot{#1}}}

\newcommand\Div{\mathop{\mathrm{div}}\nolimits}
\newcommand\Tr{\mathop{\smash{\arr{\mathrm{Tr}}}\vphantom{T}}\nolimits}
\newcommand\Trace{\mathop{\mathrm{Tr}}\nolimits}
\newcommand\M{\mathop{\smash{\arr{\mathrm{M}}}\vphantom{M}}\nolimits}

\newcommand\re{\mathop{\mathrm{Re}}\nolimits}

 \let\R\RR

\newcommand\N{\mathbb{N}} 
\newcommand\1{\mathbf{1}}
\newcommand\D{\mathcal{D}}
\newcommand\s{\mathcal{S}}

\newcommand\XX{\mathfrak{X}}
\newcommand\YY{\mathfrak{Y}}

\newcommand\pureH{\parallel}

\newcommand\dmn{{n+1}}
\newcommand\pdmn{{(n+1)}}
\newcommand\dmnMinusOne{n}

\begingroup
\makeatletter
\def\citation#1{}
\def\bibcite#1#2{}
\input Rellich-potentials.labels
\input Rellich-Neumann.labels
\global\def\eqref{\@ifstar\@eqref\@@eqref}
\global\def\@eqref#1{\textup{\tagform@{\ref*{#1}}}}
\global\def\@@eqref#1{\textup{\tagform@{\ref{#1}}}}
\endgroup

\begin{document}

\title[Trace theorems for solutions]{Dirichlet and Neumann boundary values of solutions to higher order elliptic equations}

\author{Ariel Barton}
\address{Ariel Barton, Department of Mathematical Sciences,
			309 SCEN,
			University of Ar\-kan\-sas,
			Fayetteville, AR 72701}
\email{aeb019@uark.edu}

\author{Steve Hofmann}
\address{Steve Hofmann, 202 Math Sciences Bldg., University of Missouri, Columbia, MO 65211}
\email{hofmanns@missouri.edu}
\thanks{Steve Hofmann is partially supported by the NSF grant DMS-1361701.}

\author{Svitlana Mayboroda}
\address{Svitlana Mayboroda, Department of Mathematics, University of Minnesota, Minneapolis, Minnesota 55455}
\email{svitlana@math.umn.edu}
\thanks{Svitlana Mayboroda is partially supported by the NSF CAREER Award DMS 1056004,  the NSF INSPIRE Award DMS 1344235, and the NSF Materials Research Science and Engineering Center Seed Grant.}

\begin{abstract}
We show that if $u$ is a solution to a linear elliptic differential equation of order $2m\geq 2$ in the half-space with $t$-independent coefficients, and if $u$ satisfies certain area integral estimates, then the Dirichlet and Neumann boundary values of $u$ exist and lie in a Lebesgue space $L^p(\mathbb{R}^n)$ or Sobolev space $\dot W^p_{\pm 1}(\mathbb{R}^n)$. Even in the case where $u$ is a solution to a second order equation, our results are new for certain values of~$p$.
\end{abstract}



\maketitle 

\tableofcontents

\section{Introduction}

This paper is part of an ongoing study of elliptic differential operators of the form
\begin{equation}\label{B:eqn:divergence}
Lu = (-1)^m \sum_{\abs{\alpha}=\abs{\beta}=m} \partial^\alpha (A_{\alpha\beta} \partial^\beta u)\end{equation}
for $m\geq 1$, with general bounded measurable coefficients.

Specifically, we consider boundary value problems for such operators. One such problem is the Dirichlet problem
\begin{equation}\label{B:eqn:introduction:Dirichlet}
Lu =0 \text{ in }\Omega,
\qquad
\nabla^{m-1} u= \arr f\text{ on }\partial\Omega\end{equation}
for a specified domain~$\Omega$ and array $\arr f$ of boundary functions. 

We are also interested in the corresponding higher-order Neumann problem, defined as follows. We say that  $Lu=0$ in~$\Omega$ in the weak sense if
\begin{equation*}\sum_{\abs\alpha=\abs\beta=m} \int_\Omega \partial^\alpha\varphi\,A_{\alpha\beta}\,\partial^\beta u=0
\end{equation*}
for all smooth functions $\varphi$ whose support is compactly contained in~$\Omega$.
If $\varphi$ is smooth and compactly supported in $\R^\dmn\supsetneq \Omega$, then the above integral is no longer zero; however, it depends only on $u$ and the behavior of $\varphi$ near the boundary, not the values of $\varphi$ in the interior of~$\Omega$. The Neumann problem with boundary data $\arr g$ is then the problem of finding a function $u$ such that
\begin{equation}\label{B:eqn:introduction:Neumann}
\sum_{\abs\alpha=\abs\beta=m} \int_\Omega \partial^\alpha\varphi\,A_{\alpha\beta}\,\partial^\beta u=\sum_{\abs\gamma=m-1}\int_{\partial\Omega} \partial^\gamma\varphi\,g_\gamma\,d\sigma\quad\text{for all }\varphi\in C^\infty_0(\R^\dmn).\end{equation}
In the second-order case ($m=1$), if $\mat A$ and $\nabla u$ are continuous up to the boundary, then integrating by parts reveals that $g=\nu\cdot\mat A\nabla u$, where $\nu$ is the unit outward normal vector, and so this notion of Neumann problem coincides with the more familiar Neumann problem in the second order case. 

In the higher order case, the Neumann boundary values $\arr g$ of $u$ are a linear operator on $\{\nabla^{m-1} \varphi\big\vert_{\partial\Omega}:\varphi\in C^\infty_0(\R^\dmn)\}$. 
Given a bound on the above integral in terms of, for example, $\doublebar{\nabla^{m-1}\varphi\big\vert_{\partial\Omega}}_ {L^{p'}(\partial\Omega)}$, we may extend $\arr g$ by density to a linear operator on a closed subspace of ${L^{p'}(\partial\Omega)}$; however, gradients of smooth functions are not dense in ${L^{p'}(\partial\Omega)}$, and so $\arr g$ lies not in the dual space $L^p(\partial\Omega)$ but in a quotient space of $L^p(\partial\Omega)$. We refer the interested reader to  \cite{BarM16B,BarHM15p} for further discussion of the nature of higher order Neumann boundary values.

In this paper we will focus on trace results. That is, for a specific class of coefficients~$\mat A$, given a solution $u$ to $Lu=0$ in the upper half-space, and given that a certain norm of $u$ is finite, we will show that the Dirichlet and Neumann boundary values exist, and will produce estimates on the Dirichlet and Neumann boundary values $\arr f$ and $\arr g$ in formulas~\eqref{B:eqn:introduction:Dirichlet} or~\eqref{B:eqn:introduction:Neumann}; specifically, we will find norms of $u$ that force $\arr f$ and $\arr g$ to lie in Lebesgue spaces $L^p(\partial\R^\dmn_+)$ or Sobolev spaces $\dot W^p_{\pm 1}(\partial\R^\dmn_+)$.

These results may be viewed as a converse to the well-posedness results central to the theory; that is, well-posedness results begin with the boundary values $\arr f$ or $\arr g$ and attempt to construct functions $u$ that satisfy the problems~\eqref{B:eqn:introduction:Dirichlet} or~\eqref{B:eqn:introduction:Neumann}. 


We now turn to the specifics of our results.

We will consider solutions $u$ to $Lu=0$ in the upper half-space $\R^\dmn_+$, where $L$ is an operator of the form \eqref{B:eqn:divergence}, with coefficients that are $t$-indepen\-dent in the sense that
\begin{equation}\label{B:eqn:t-independent}\mat A(x,t)=\mat A(x,s)=\mat A(x) \quad\text{for all $x\in\R^n$ and all $s$, $t\in\R$}.\end{equation}
At least in the case of well-posedness results, it has long been known (see \cite{CafFK81}) that some regularity of the coefficients $\mat A$ in formula~\eqref{B:eqn:divergence} is needed. Many important results in the second order theory have been proven in the case of $t$-independent coefficients in the half-space; see, for example, \cite{KenR09,AlfAAHK11,AusAH08,AusAM10,Bar13,AusM14,HofKMP15A,HofKMP15B,HofMitMor15,BarM16A}. The $t$-independent case may also be used as a starting point for certain $t$-dependent perturbations; see, for example, \cite{KenP93,KenP95,AusA11,HofMayMou15}.
In the higher-order case, well posedness of the Dirichlet problem for certain fourth-order differential operators (of a strange form, that is, not of the form~\eqref{B:eqn:divergence}) with $t$-independent coefficients was established in \cite{BarM13}. The theory of boundary value problems for $t$-independent operators of the form \eqref{B:eqn:divergence} is still in its infancy; 
the authors of the present paper have begun its study in the papers \cite{BarHM15p,BarHM17pA} and intend to continue its study in the present paper, in \cite{BarHM17pC}, and in future work.

We will be interested in solutions that satisfy bounds in terms of the Lusin area integral  $\mathcal{A}_2$ given by
\begin{equation}\label{B:eqn:lusin:intro}
\mathcal{A}_2 H(x) = \biggl(\int_0^\infty \int_{\abs{x-y}<t} \abs{H(y,t)}^2\frac{dy\,dt}{\abs{t}^\dmn}\biggr)^{1/2} \quad\text{for $x\in\R^n$}
.\end{equation}

Our main results may be stated as follows.
\begin{thm} \label{B:thm:intro}
Suppose that $L$ is an elliptic operator associated with coefficients $\mat A$ that are $t$-independent in the sense of formula~\eqref{B:eqn:t-independent} and satisfy the ellipticity conditions \eqref{B:eqn:elliptic} and~\eqref{B:eqn:elliptic:bounded}.

If $Lu=0$ in $\R^\dmn_+$, denote the Dirichlet and Neumann boundary values of $u$ by $\Tr_{m-1}^+ u$ and $\M_{\mat A}^+ u$, respectively.

There exist some constants $\varepsilon_1>0$ and $\varepsilon_2>0$ depending only on the dimension $\dmn$ and the constants $\lambda$ and $\Lambda$ in the bounds \eqref{B:eqn:elliptic} and~\eqref{B:eqn:elliptic:bounded} such that the following statements are valid. 
(If $\dmn=2$ or $\dmn=3$ then $\varepsilon_1=\infty$.)

Let $v$ and $w$ be functions defined in $\R^\dmn_+$ such that $Lv=Lw=0$ in $\R^\dmn_+$. 
Suppose that $\mathcal{A}_2(t\nabla^m v)\in L^p(\R^n)$ and $\mathcal{A}_2(t\nabla^m \partial_t w)\in L^p(\R^n)$ for some $1<p<\infty$. If $p>2$, assume in addition that $\nabla^m v\in L^2(\R^n\times(\sigma,\infty))$ and $\nabla^m \partial_\dmn w\in L^2(\R^n\times(\sigma,\infty))$ for all $\sigma>0$. (It is acceptable if the $L^2$ norm approaches infinity as $\sigma\to 0^+$.)

If $p$ lies in the range indicated below, then there exists a constant array $\arr c$ and a function $\tilde w$, with  $L\tilde w=0$ and  $\nabla^m \partial_\dmn \tilde w = \nabla^m \partial_\dmn w$ in $\R^\dmn$, such that the Dirichlet and Neumann boundary values of $ v$ and $\tilde w$ exist in the sense of formulas~\eqref{B:eqn:Dirichlet} and \eqref{B:eqn:Neumann:E} and satisfy the bounds
\begin{align}
\label{B:eqn:introduction:Dirichlet:1}
\doublebar{\Tr_{m-1}^+ v-\arr c}_{L^p(\R^n)} 
&\leq C_p\doublebar{\mathcal{A}_2(t\nabla^m v)}_{L^p(\R^n)}
,&&1<p\leq 2+\varepsilon_1
,\\
\label{B:eqn:introduction:Neumann:1}
\doublebar{\M_{\mat A}^+ v}_{\dot W_{-1}^p(\R^n)} 
&\leq C_p\doublebar{\mathcal{A}_2(t\nabla^m v)}_{L^p(\R^n)}
,&&1<p< \infty
,\\
\label{B:eqn:introduction:Dirichlet:2}
\doublebar{\Tr_{m-1}^+\tilde w}_{\dot W_1^p(\R^n)} 
&\leq C_p\doublebar{\mathcal{A}_2(t\nabla^m \partial_t w)}_{L^p(\R^n)}
,&&1<p\leq 2+\varepsilon_2
,\\
\label{B:eqn:introduction:Neumann:2}
\doublebar{\M_{\mat A}^+\tilde w}_{L^p(\R^n)} 
&\leq C_p\doublebar{\mathcal{A}_2(t\nabla^m \partial_t w)}_{L^p(\R^n)}
,&&1<p\leq 2+\varepsilon_2
.\end{align}


Define
\begin{equation}\label{B:eqn:w:bound}
W_{p,q}(\tau)=\biggl(\int_{\R^n} \biggl(\fint_{B((x,\tau),\tau/2)} \abs{\nabla^m w}^q dx\biggr)^{p/q}\biggr)^{1/p}.
\end{equation}

If for some $q>0$ and some $\tau>0$ we have that $W_{p,q}(\tau)<\infty$,
then $\nabla^m w=\nabla^m \tilde w$. 

If for some $q>0$ we have that $W_{p,q}(\tau)$ is bounded uniformly in $\tau>0$, then we have the estimate
\begin{equation}
\label{B:eqn:introduction:Neumann:2:infinity}
\doublebar{\M_{\mat A}^+ w}_{L^p(\R^n)} 
\leq C_p\doublebar{\mathcal{A}_2(t\nabla^m \partial_t w)}_{L^p(\R^n)} +C_{p,q}\sup_{\tau>0}W_{p,q}(\tau),\quad 1<p<\infty
.\end{equation}
\end{thm}
Here the $L^p$ and $\dot W^p_{-1}$ norms of the Neumann boundary values are meant in the sense of operators on (not necessarily dense) subspaces of $L^{p'}$ and $\dot W^{p'}_1$, that is, in the sense that
\begin{align*}
\doublebar{\M_{\mat A} v}_{\dot W_{-1}^p(\R^n)}
&=\sup_{\varphi\in C^\infty_0(\R^\dmn)}
\frac{
\abs{\langle \nabla^{m-1} \varphi(\,\cdot\,,0),\M_{\mat A} v\rangle_{\R^n}}}{\doublebar{\nabla^{m-1}\varphi(\,\cdot\,,0)}_{\dot W_1^{p'}(\R^n)}}
,\\
\doublebar{\M_{\mat A} w}_{L^p(\R^n)}
&=\sup_{\varphi\in C^\infty_0(\R^\dmn)}
\frac{
\abs{\langle \nabla^{m-1} \varphi(\,\cdot\,,0),\M_{\mat A} \tilde w\rangle_{\R^n}}}{C\doublebar{\nabla^{m-1}\varphi(\,\cdot\,,0)}_{L^{p'}(\R^n)} }
.\end{align*}

These results are new in the higher order case. In the second order case, the bounds \eqref{B:eqn:introduction:Dirichlet:1}--\eqref{B:eqn:introduction:Neumann:2} are known in the case $p=2$, but are new for certain other values of~$p$.

Specifically, if $L$ and $L^*$ are second order operators that satisfy the De Giorgi-Nash-Moser condition, then the bound \eqref{B:eqn:introduction:Neumann:2:infinity} is new in the case $p>2+\varepsilon$ and the bounds \eqref{B:eqn:introduction:Dirichlet:1} and~\eqref{B:eqn:introduction:Neumann:1} are new in the case $1<p<2-\varepsilon$, where $\varepsilon$ is a positive number related to the exponent in the De Giorgi-Nash condition. If $L$ is an arbitrary second order operator (that is, without the De Giorgi-Nash-Moser condition), then the 
bound~\eqref{B:eqn:introduction:Dirichlet:1} is new in the case $1<p<2$, the bound \eqref{B:eqn:introduction:Dirichlet:2} is new in the case $1<p<2n/(n+2)$, the bound \eqref{B:eqn:introduction:Neumann:1} is new in the cases $1<p<2$ and $2n/(n-2)<p<\infty$, and the bound~\eqref{B:eqn:introduction:Neumann:2} is new in the cases $1<p<2n/(n+2)$ and $2<p<\infty$.


\begin{rmk}
Let $\widetilde N H(x) = \sup\{ \bigl(\fint_{B((y,t),t/2)} \abs{H}^2\bigr)^{1/2}:\abs{x-y}<t\}$ be the modified nontangential maximal function introduced in \cite{KenP93}.
%
Estimates of the form $\doublebar{\widetilde N (\nabla^{m-1}u)}_{L^p(\R^n)}\approx \doublebar{\mathcal{A}_2(t\nabla^m u)}_{L^p(\R^n)}$ for a solution $u$ to $Lu=0$ have played an important role in the theory of boundary value problems. See \cite{Dah80a,DahJK84,DahKPV97,HofKMP15A} for proofs of this equivalence under various assumptions on~$L$. 

This equivalence can be used to solve boundary value problems. In \cite{HofKMP15A}, the authors used this equivalence together with the method of $\varepsilon$-approximability of \cite{KenKPT00} to establish well posedness of the Dirichlet problem with $L^p$ boundary data for second order operators with $t$-independent coefficients. In the higher order case, in \cite{She06B,KilS11A} Kilty and Shen have used this equivalence to prove well posedness of the $L^p$-Dirichlet and $\dot W^q_1$-Dirichlet problems for constant coefficient operators and various ranges of $p$ and~$q$, and in \cite{Ver96} Verchota used this equivalence to prove a maximum principle in three-dimensional Lipschitz domains for constant coefficient elliptic systems.

The results of the present paper constitute a major first step towards proving an estimate of the form $\doublebar{\widetilde N (\nabla^{m-1}u)}_{L^p(\R^n)}\leq C \doublebar{\mathcal{A}_2(t\nabla^m u)}_{L^p(\R^n)}$. Specifically, if $Lu=0$ in $\R^\dmn_+$ and $\nabla^m u\in L^2(\R^\dmn_+)$, then we will see (formula~\eqref{B:eqn:green} below) that 
\[\nabla^m u=-\nabla^m \D^{\mat A}(\Tr_{m-1}^+  u) + \nabla^m \s^{L}(\M_{\mat A}^+  u) \]
where $\D^{\mat A}$ and $\s^L$ denote the double and single layer potentials. This Green's formula will be extended to solutions $u$ that satisfy $\mathcal{A}_2(t\nabla^m u)\in L^2(\R^n)$ in \cite{BarHM17pC}. In a forthcoming paper \cite{BarHM17pD}, we intend to extend the Green's formula to solutions $u$ with $\mathcal{A}_2(t\nabla^m u)\in L^p(\R^n)$, and to show that the double and single layer potentials satisfy nontangential estimates; combined with Theorem~\ref{B:thm:intro}, this implies the desired estimate $\doublebar{\widetilde N (\nabla^{m-1}u)}_{L^p(\R^n)}\leq C \doublebar{\mathcal{A}_2(t\nabla^m u)}_{L^p(\R^n)}$.


\end{rmk}

We mention some refinements to Theorem~\ref{B:thm:intro}. 

The definition \eqref{B:eqn:Neumann:E} below of Neumann boundary values is somewhat delicate; a more robust formulation of $\M_{\mat A}^+ w$ is stated in Theorem~\ref{B:thm:Neumann:2}. (The delicate formulation is necessary to contend with $v$ in the full generality of Theorem~\ref{B:thm:intro}; however, if $v$ satisfies some additional regularity assumptions, such as $\nabla^m v\in L^2(\R^\dmn_+)$, then the formulation of Neumann boundary values of formula~\eqref{B:eqn:Neumann:E} coincides with more robust formulations. See Section~\ref{B:sec:dfn:Neumann}.)

There is some polynomial $P$ of degree $m-1$ such that $\nabla^{m-1}P=\arr c$. Clearly $\tilde v=v-P$ is also a solution to $L\tilde v=0$ in $\R^\dmn_+$, and $\nabla^m\tilde v=\nabla^m v$ and so $\tilde v$ satisfies the same estimates as~$v$, and furthermore $\M_{\mat A}\tilde v=\M_{\mat A} v$.

Some additional bounds on $\tilde w$ and $\tilde v=v-P$ are stated in Theorems~\ref{B:thm:Dirichlet:1} and~\ref{B:thm:Dirichlet:2}. In particular, we have that
\begin{align*}
\sup_{t>0}\doublebar{\nabla^{m-1} v(\,\cdot\,,t)-\arr c}_{L^p(\R^n)} 
&\leq C_p\doublebar{\mathcal{A}_2(t\nabla^m v)}_{L^p(\R^n)}
,\\
\sup_{t>0}\doublebar{\nabla^{m} \tilde w(\,\cdot\,,t)}_{L^p(\R^n)}
&\leq C_p\doublebar{\mathcal{A}_2(t\nabla^m \partial_t w)}_{L^p(\R^n)}
\end{align*}
and the limits 
\begin{align*}
\lim_{T\to\infty}\doublebar{\nabla^{m-1} v(\,\cdot\,,T)-\arr c}_{L^p(\R^n)} 
&+
\lim_{t\to 0^+}\doublebar{\nabla^{m-1} v(\,\cdot\,,t)-\Tr_{m-1}^+ v}_{L^p(\R^n)} 
=0
,\\
\lim_{T\to\infty}\doublebar{\nabla^{m} \tilde w(\,\cdot\,,T)}_{L^p(\R^n)} 
&+
\lim_{t\to 0^+}\doublebar{\nabla^{m} \tilde w(\,\cdot\,,t)-\Tr_m^+\tilde w}_{L^p(\R^n)} 
=0
\end{align*}
are valid. Notice that an $L^p$ bound on $\nabla^m \tilde w(\,\cdot\,,t)$ is stronger than a $\dot W^p_1$ bound on $\nabla^{m-1} \tilde w(\,\cdot\,,t)$, as the former involves estimates on all derivatives of order $m-1$ while the latter involves only derivatives at least one component of which are tangential to the boundary.


It is clear that $W_{p,p}(\tau)\leq C\sup_{t>0}\doublebar{\nabla^{m}  w(\,\cdot\,,t)}_{L^p(\R^n)}$. In addition, we remark that $W_{p,2}(\tau)\leq \doublebar{\widetilde N(\nabla^m w)}_{L^p(\R^n)}^p$, where $\widetilde N$ is the modified nontangential maximal function introduced in \cite{KenP93}.

We now review the history of such results. The theory of boundary values of harmonic functions may be said to begin with Fatou's celebrated result \cite{Fat06} that if  $f$ is holomorphic in the upper half-plane, and if $\sup_{t>0}\doublebar{f(\,\cdot\,,t)}_{L^p(\R)}$ is finite, then $f$ has Dirichlet boundary values in the sense that there is some $f_0\in L^p(\R)$ such that as $t\to 0^+$, we have $\doublebar{f(\,\cdot\,,t)-f_0}_{L^p(\R)} \to 0$ and $f(x,t)\to f_0(x)$ for almost every $x\in \R$.


If $p\geq 2$, then this result may be extended from holomorphic functions in $\R^2_+$ to harmonic functions in $\R^\dmn_+$ for $n\geq 1$; see \cite{Cal50,Car62}. Furthermore, with slightly stronger assumptions the same result (with $1\leq p\leq \infty$) is true for functions harmonic in Lipschitz domains; see \cite{HunW68,HunW70}.
Specifically, 
let $N$ be the nontangential maximal operator
\begin{equation*}
NH(X)=\sup_{\{Y\in\Omega:\abs{X-Y}<\sqrt{2}\delta(Y)\}} \abs{H(Y)}\quad\text{for $X\in\partial\Omega$.}\end{equation*}
If $u$ is harmonic in a Lipschitz domain~$\Omega$, then at almost every $X\in\partial\Omega$ (with respect to harmonic measure) for which $Nu(X)<\infty$, a pointwise nontangential limit exists (that is, $\lim_{Y\to X} u(Y)$ exists provided we consider only $Y$ in the nontangential cone ${\{Y\in\Omega:\abs{X-Y}<\sqrt{2}\delta(Y)\}}$).
Sets of harmonic measure zero on boundaries of Lipschitz domains have surface measure zero; see \cite{Dah77}.
In particular, if $Nu\in L^p(\partial\Omega)$ then $u$ has nontangential limits almost everywhere; if $f(X)=\lim_{Y\to X\>\text{n.t.}}u(Y)$ then necessarily $\abs{f(X)}\leq Nu(X)$.

In \cite{Dah80a}, Dahlberg showed that if $u$ is harmonic in a bounded Lipschitz domain $\Omega\subset\R^\dmn$, then if $u$ is normalized appropriately we have that
\begin{equation}\label{B:eqn:lusin:nontangential}\doublebar{\mathcal{A}_2^\Omega(\delta\nabla u)}_{L^p(\partial\Omega)}\approx\doublebar{Nu}_{L^p(\partial\Omega)}
,\qquad 0< p< \infty
\end{equation}
where $\mathcal{A}_2^\Omega$ is a variant on the Lusin area integral of formula~\eqref{B:eqn:lusin:intro} appropriate to the domain~$\Omega$.
Thus, Dahlberg's results imply the analogue to the bound \eqref{B:eqn:introduction:Dirichlet:1} (for $0<p<\infty$) in Lipschitz domains for harmonic functions~$v$. Because the gradient of a harmonic function is harmonic, Dahlberg's results also imply the Lipschitz analogue to the bounds \eqref{B:eqn:introduction:Dirichlet:2} and \eqref{B:eqn:introduction:Neumann:2} (with Neumann boundary values $\nu\cdot \nabla w$) for harmonic functions.

Turning to more general second order operators, in \cite{CafFMS81} the results of \cite{HunW68,HunW70} for nontangentially bounded harmonic functions were generalized to the case of nontangentially bounded solutions to $\Div \mat A\nabla u=0$ where $\mat A$ is a real-valued matrix.
The equivalence \eqref{B:eqn:lusin:nontangential} was established in \cite{DahJK84} for such~$u$, provided that the Dirichlet problem with boundary data in $L^p(\partial\Omega)$ is well posed for at least one $p$ with $1<p<\infty$. If the Dirichlet problem is well-posed then $L$-harmonic measure is absolutely continuous with respect to surface measure. Thus, for such coefficients the analogue to the bound \eqref{B:eqn:introduction:Dirichlet:1}, in Lipschitz domains, and for $1<p<\infty$, is valid.

In \cite[Section~3]{KenP93} 
it was shown that if $\Div \mat A\nabla w=0$ in the unit ball, where $\mat A$ is real, and if $\widetilde N(\nabla w)\in L^p(\partial\Omega)$ for $1<p<\infty$, where $\widetilde N$ is a suitable modification of the nontangential maximal function, then the Dirichlet boundary values $w\big\vert_{\partial\Omega}$ lie in the boundary Sobolev space $\dot W^p_1(\partial\Omega)$ and the Neumann boundary values $M_{\mat A}^\Omega w = \nu\cdot \mat A\nabla w$ lie in $L^p(\partial\Omega)$. With some modifications, the requirement that $\mat A$ be real-valued may be dropped (and indeed the same argument is valid for higher order operators). These results are the analogues to the bounds  \eqref{B:eqn:introduction:Dirichlet:2} and~\eqref{B:eqn:introduction:Neumann:2} with nontangential estimates in place of area integral estimates.

Turning to the case of complex coefficients, or the case where well-posedness of the Dirichlet problem is {not} assumed, in \cite[Theorem~2.3(i), (iii)]{AusA11}, the equivalence 
\begin{equation}
\label{B:eqn:lusin:nontangential:2}
\doublebar{\mathcal{A}_2(t\nabla \partial_t w)}_{L^2(\R^n)} \approx
\doublebar{\widetilde N(\nabla w)}_{L^2(\R^n)}
\end{equation}
for solutions $w$ to elliptic equations with $t$-independent coefficients was established; combined with the arguments of \cite{KenP93}, this yields the bounds \eqref{B:eqn:introduction:Dirichlet:2} and~\eqref{B:eqn:introduction:Neumann:2} for $p=2$ and $m=1$. (Under some further assumptions, this equivalence was established in \cite{AlfAAHK11}.) Furthermore, in \cite[Theorem~2.4(i)]{AusA11} the bound \eqref{B:eqn:introduction:Dirichlet:1} was established for general $t$-independent coefficients, again for $p=2$ and $m=1$. These results extend to $t$-dependent operators that satisfy a small (or finite) Carleson norm condition.

The result \eqref{B:eqn:introduction:Neumann:2}, and indeed the Neumann problem with boundary data in negative smoothness spaces, has received little attention to date; most of the known results involve the Neumann problem for inhomogeneous differential equations and the related theory of Neumann boundary value problems with data in fractional smoothness spaces \cite{FabMM98,Zan00,Agr07,Agr09,MitM13A,MitM13B,BarM16A}. However, the Neumann problem with boundary data in the negative Sobolev space $\dot W^p_{-1}(\partial\R^\dmn_+)$ was investigated in \cite{AusM14}; furthermore, as a consequence of \cite[Theorems~1.1--1.2]{AusS14p}, we have the bound \eqref{B:eqn:introduction:Neumann:1} with $m=1$ and $2-\varepsilon<p<2n/(n-2)+\varepsilon$, where $\varepsilon>0$ depends on~$L$, as well as improved ranges of~$p$ for the bounds \eqref{B:eqn:introduction:Dirichlet:1}, \eqref{B:eqn:introduction:Dirichlet:2} and~\eqref{B:eqn:introduction:Neumann:2} with $m=1$. 
(Specifically, the bound \eqref{B:eqn:introduction:Dirichlet:1} was also established for $2-\varepsilon<p<2n/(n-2)+\varepsilon$, and the bounds \eqref{B:eqn:introduction:Dirichlet:2} and~\eqref{B:eqn:introduction:Neumann:2} were established for $2n/(n+2)-\varepsilon<p<2+\varepsilon$. In the case when  $L$ and $L^*$ satisfy the De Giorgi-Nash-Moser estimates, and in a few other special cases, the estimates are valid in the ranges $1<p<2+\varepsilon$  and $2-\varepsilon<p<\infty$. 




We remark that Fatou's theorem, our Theorem~\ref{B:thm:intro}, and many of the other results discussed above, are valid only for solutions to elliptic equations. An arbitrary function that satisfies square function estimates or nontangential bounds need not have a limit at the boundary in any sense. Many of the trace results applied in the higher order theory have been proven in much higher generality. It is well known that if $u$ is \emph{any} function in the Sobolev space $\dot W^p_m(\Omega)$, where $\Omega$ is a bounded Lipschitz domain, $1<p<\infty$ and $m\geq 1$ is an integer, then $\Tr_{m-1}^\Omega u$ lies in the Besov space $\dot B^{p,p}_{1-1/p}(\partial\Omega)$. Similar results are true if $u$ lies in a Besov space $\dot B^{p,q}_{s+1/p}(\Omega)$ (see \cite{JonW84}) or a weighted Sobolev space (see \cite{MazMS10,Kim07,Bar16pB}). These results all yield that the boundary values $\Tr_{m-1}^\Omega u$ lie in a boundary Besov space $\dot B^{p,p}_s(\partial\Omega)$, with smoothness parameter $s$ satisfying $0<s<1$.

Such results, and their (i.e., extension results) have been used to pass between the Dirichlet problem for a homogeneous differential equation and the Dirichlet problem with homogeneous boundary data, that is, between the problems
\begin{align}
\label{B:eqn:Dirichlet:inside}
L u&=H \text{ in }\Omega,
&\nabla^{m-1} u&=0\text{ on }\partial\Omega, 
&\doublebar{ u}_\XX&\leq C \doublebar{ H}_\YY,
\\
\label{B:eqn:Dirichlet:boundary}
L u&=0 \text{ in }\Omega,
&\nabla^{m-1} u&=\arr f \text{ on }\partial\Omega, 
&\doublebar{ u}_\XX&\leq C \doublebar{\arr f}_{\dot B^{p,p}_s(\partial\Omega)}
\end{align}
for some appropriate spaces $\XX$ and~$\YY$. See \cite{AdoP98, Agr07, MazMS10, MitMW11, MitM13A, MitM13B, BreMMM14, Bar16pA, Bar16pB}.

We are interested in the case where the boundary data lies in a Lebesgue space or Sobolev space, that is, where the smoothness parameter is an integer. In this case the obvious associated inhomogeneous problem is ill-posed, even in very nice cases (for example, for harmonic functions in the half-space) and so the arguments involving the inhomogeneous problem~\eqref{B:eqn:Dirichlet:inside} are not available. Furthermore, in this case it generally is necessary to exploit the fact that $u$ is a solution to an elliptic equation, and so the method of proof of Theorem~\ref{B:thm:intro} is completely different.

The outline of this paper is as follows. In Section~\ref{B:sec:dfn} we will define the terminology we will use throughout the paper. In Section~\ref{B:sec:known} we will summarize some known results of the theory of higher order elliptic equations. 
In Section~\ref{B:sec:preliminaries} we will prove a few results that will be of use in both Sections~\ref{B:sec:Dirichlet} and~\ref{B:sec:Neumann}. In particular, we will prove Lemma~\ref{B:lem:control:lusin}, the technical core of our paper.
Finally, we will prove our results concerning Dirichlet boundary values in Section~\ref{B:sec:Dirichlet}, and our results concerning Neumann boundary values in Section~\ref{B:sec:Neumann}; these results will be stated as Theorems~\ref{B:thm:Dirichlet:1}, \ref{B:thm:Dirichlet:2}, \ref{B:thm:Neumann:1} and~\ref{B:thm:Neumann:2}.
We mention that many of the ideas in the present paper come from the proof of the main estimate (3.9) of \cite{HofKMP15B}. The results of the present paper allow for a slightly different approach to proving the results of \cite{HofKMP15B}; see \cite[Remark~\ref*{C:rmk:HofKMP15B}]{BarHM17pC}.

\subsection*{Acknowledgements}
We would like to thank the American Institute of Mathematics for hosting the SQuaRE workshop on ``Singular integral operators and solvability of boundary problems for elliptic equations with rough coefficients,'' and the Mathematical Sciences Research Institute for hosting a Program on Harmonic Analysis,  at which many of the results and techniques of this paper were discussed.

\section{Definitions}\label{B:sec:dfn}

In this section, we will provide precise definitions of the notation and concepts used throughout this paper. 

We mention that throughout this paper, we will work with elliptic operators~$L$ of order~$2m$ in the divergence form \eqref{B:eqn:divergence} acting on functions defined on~$\R^\dmn$.
As usual, we let $B(X,r)$ denote the ball in $\R^n$ of radius $r$ and center $X$. We let $\R^\dmn_+$ and $\R^\dmn_-$ denote the upper and lower half-spaces $\R^n\times (0,\infty)$ and $\R^n\times(-\infty,0)$; we will identify $\R^n$ with $\partial\R^\dmn_\pm$.


If $Q\subset\R^n$ is a cube, we let $\ell(Q)$ be its side-length, and we let $cQ$ be the concentric cube of side-length $c\ell(Q)$. If $E$ is a set of finite measure, we let $\fint_E f(x)\,dx=\frac{1}{\abs{E}}\int_E f(x)\,dx$.

\subsection{Multiindices and arrays of functions}

We will reserve the letters $\alpha$, $\beta$, $\gamma$, $\zeta$ and~$\xi$ to denote multiindices in $\N^\dmn$. (Here $\N$ denotes the nonnegative integers.) If $\zeta=(\zeta_1,\zeta_2,\dots,\zeta_\dmn)$ is a multiindex, then we define $\abs{\zeta}$, $\partial^\zeta$ and $\zeta!$ in the usual ways, as $\abs{\zeta}=\zeta_1+\zeta_2+\dots+\zeta_\dmn$, $\partial^\zeta=\partial_{x_1}^{\zeta_1}\partial_{x_2}^{\zeta_2} \cdots\partial_{x_\dmn}^{\zeta_\dmn}$,
and $\zeta!=\zeta_1!\,\zeta_2!\cdots\zeta_\dmn!$. 

We will routinely deal with arrays $\arr F=\begin{pmatrix}F_{\zeta}\end{pmatrix}$ of numbers or functions indexed by multiindices~$\zeta$ with $\abs{\zeta}=k$ for some~$k\geq 0$.
In particular, if $\varphi$ is a function with weak derivatives of order up to~$k$, then we view $\nabla^k\varphi$ as such an array.

The inner product of two such arrays of numbers $\arr F$ and $\arr G$ is given by
\begin{equation*}\bigl\langle \arr F,\arr G\bigr\rangle =
\sum_{\abs{\zeta}=k}
\overline{F_{\zeta}}\, G_{\zeta}.\end{equation*}
If $\arr F$ and $\arr G$ are two arrays of functions defined in a set $\Omega$ in Euclidean space, then the inner product of $\arr F$ and $\arr G$ is given by
\begin{equation*}\bigl\langle \arr F,\arr G\bigr\rangle_\Omega =
\sum_{\abs{\zeta}=k}
\int_{\Omega} \overline{F_{\zeta}(X)}\, G_{\zeta}(X)\,dX.\end{equation*}


We let $\vec e_j$ be the unit vector in $\R^\dmn$ in the $j$th direction; notice that $\vec e_j$ is a multiindex with $\abs{\vec e_j}=1$. We let $\arr e_{\zeta}$ be the ``unit array'' corresponding to the multiindex~$\zeta$; thus, $\langle \arr e_{\zeta},\arr F\rangle = F_{\zeta}$.

We will let $\nabla_\pureH$ denote either the gradient in~$\R^n$, or the $n$ horizontal components of the full gradient~$\nabla$ in $\R^\dmn$. (Because we identify $\R^n$ with $\partial\R^\dmn_\pm\subset\R^\dmn$, the two uses are equivalent.) If $\zeta$ is a multiindex with $\zeta_\dmn=0$, we will occasionally use the terminology $\partial_\pureH^\zeta$ to emphasize that the derivatives are taken purely in the horizontal directions.

\subsection{Elliptic differential operators}

Let $\mat A = \begin{pmatrix} A_{\alpha\beta} \end{pmatrix}$ be a matrix of measurable coefficients defined on $\R^\dmn$, indexed by multtiindices $\alpha$, $\beta$ with $\abs{\alpha}=\abs{\beta}=m$. If $\arr F$ is an array, then $\mat A\arr F$ is the array given by
\begin{equation*}(\mat A\arr F)_{\alpha} =
\sum_{\abs{\beta}=m}
A_{\alpha\beta} F_{\beta}.\end{equation*}

We will consider coefficients that satisfy the G\r{a}rding inequality
\begin{align}
\label{B:eqn:elliptic}
\re {\bigl\langle\nabla^m \varphi,\mat A\nabla^m \varphi\bigr\rangle_{\R^\dmn}}
&\geq
	\lambda\doublebar{\nabla^m\varphi}_{L^2(\R^\dmn)}^2
	\quad\text{for all $\varphi\in\dot W^2_m(\R^\dmn)$}
\end{align}
and the bound
\begin{align}
\label{B:eqn:elliptic:bounded}
\doublebar{\mat A}_{L^\infty(\R^\dmn)}
&\leq
	\Lambda
\end{align}
for some $\Lambda>\lambda>0$.
In this paper we will focus exclusively on coefficients that are $t$-inde\-pen\-dent, that is, that satisfy formula~\eqref{B:eqn:t-independent}.

We let $L$ be the $2m$th-order divergence-form operator associated with~$\mat A$. That is, we say that $L u=0$ in~$\Omega$ in the weak sense if, for every $\varphi$ smooth and compactly supported in~$\Omega$, we have that
\begin{equation}
\label{B:eqn:L}
\bigl\langle\nabla^m\varphi, \mat A\nabla^m u\bigr\rangle_\Omega
=\sum_{\abs{\alpha}=\abs{\beta}=m}
\int_{\Omega}\partial^\alpha \bar \varphi\, A_{\alpha\beta}\,\partial^\beta u
=
0
.
\end{equation}

Throughout the paper we will let $C$ denote a constant whose value may change from line to line, but which depends only on the dimension $\dmn$, the ellipticity constants $\lambda$ and $\Lambda$ in the bounds \eqref{B:eqn:elliptic} and~\eqref{B:eqn:elliptic:bounded}, and the order~$2m$ of our elliptic operators. Any other dependencies will be indicated explicitly. 

We let  $\mat A^*$ be the adjoint matrix, that is, $A^*_{\alpha\beta}=\overline{A_{\beta\alpha}}$. We let $L^*$ be the associated elliptic operator.


\subsection{Function spaces and boundary data}

Let $\Omega\subseteq\R^n$ or $\Omega\subseteq\R^\dmn$ be a measurable set in Euclidean space. We will let $L^p(\Omega)$ denote the usual Lebesgue space with respect to Lebesgue measure with norm given by
\begin{equation*}\doublebar{f}_{L^p(\Omega)}=\biggl(\int_\Omega \abs{f(x)}^p\,dx\biggr)^{1/p}.\end{equation*}

If $\Omega$ is a connected open set, then we let the homogeneous Sobolev space $\dot W^p_m(\Omega)$ be the space of equivalence classes of functions $u$ that are locally integrable in~$\Omega$ and have weak derivatives in $\Omega$ of order up to~$m$ in the distributional sense, and whose $m$th gradient $\nabla^m u$ lies in $L^p(\Omega)$. Two functions are equivalent if their difference is a polynomial of order~$m-1$.
We impose the norm 
\begin{equation*}\doublebar{u}_{\dot W^p_m(\Omega)}=\doublebar{\nabla^m u}_{L^p(\Omega)}.\end{equation*}
Then $u$ is equal to a polynomial of order $m-1$ (and thus equivalent to zero) if and only if its $\dot W^p_m(\Omega)$-norm is zero. 
We let $L^p_loc(\Omega)$ and $\dot W^p_{k,loc}(\Omega)$ denote functions that lie in $L^p(U)$ (or whose gradients lie in $L^p(U)$) for any bounded open set $U$ with $\overline U\subsetneq\Omega$.

We will need a number of more specialized function spaces.


We will consider functions $u$ defined in $\R^\dmn_\pm$ that lie in tent spaces. 
If $x\in\R^n$ and $a\in\R$, $a\neq 0$, then let $\Gamma_a(x)=\{(y,t)\in\R^\dmn:\abs{x-y}<at\}$. Notice that $\Gamma_a(x)\subset\R^\dmn_+$ if $a>0$ and $\Gamma_a(x)\subset\R^\dmn_-$ if $a<0$. Let
\begin{equation}\label{B:eqn:lusin}
\mathcal{A}_2^a H(x) = \biggl( \int_{\Gamma_a(x)} \abs{H(y,t)}^2\frac{dy\,dt}{\abs{t}^\dmn}\biggr)^{1/2}.
\end{equation}
We will employ the shorthand $\mathcal{A}_2^-=\mathcal{A}_2^{-1}$ and $\mathcal{A}_2=\mathcal{A}_2^+=\mathcal{A}_2^1$.
If the letter $t$ appears in the argument of $\mathcal{A}_2^a$, then it denotes the coordinate function in the $t$-direction.


The case $p=2$ will be of great importance to us; we remark that if $p=2$, then 
\begin{equation}
\label{B:eqn:Lusin:2:norm}
\doublebar{\mathcal{A}_2 H}_{L^2(\R^n)}
=\biggl(c_n  \int_{0}^\infty \int_{\R^n}\abs{H(y,t)}^2\frac{dy\,dt}{t}\biggr)^{1/2}
\end{equation}
where $c_n$ is the volume of the unit disc in $\R^n$.

\subsubsection{Dirichlet boundary data and spaces}

If $u$ is defined in $\R^\dmn_+$, we let its Dirichelt boundary values be, loosely, the boundary values of the gradient $\nabla^{m-1} u$. More precisely, we let the Dirichlet boundary values be the array of functions $\Tr_{m-1}u=\Tr_{m-1}^+ u$, indexed by multiindices $\gamma$ with $\abs\gamma=m-1$, and given by
\begin{equation}
\label{B:eqn:Dirichlet}
\begin{pmatrix}\Tr_{m-1}^+  u\end{pmatrix}_{\gamma}
=f \quad\text{if}\quad
\lim_{t\to 0^+} \doublebar{\partial^\gamma u(\,\cdot\,,t)-f}_{L^1(K)}=0
\end{equation}
for all compact sets $K\subset\R^n$. If $u$ is defined in $\R^\dmn_-$, we define $\Tr_{m-1}^- u$ similarly. We remark that if $\nabla^m u\in L^1(K\times(0,\sigma))$ for any such $K$ and some $\sigma>0$, then $\Tr_{m-1}^+u$ exists, and furthermore $\begin{pmatrix}\Tr_{m-1}^+  u\end{pmatrix}_{\gamma}=\Trace \partial^\gamma u$ where $\Trace$ denotes the traditional trace in the sense of Sobolev spaces.

We will be concerned with boundary values in Lebesgue or Sobolev spaces. However, observe that the different components of $\Tr_{m-1}u$ arise as derivatives of a common function, and thus must satisfy certain compatibility conditions. We will define the Whitney spaces of functions that satisfy these compatibility conditions and have certain smoothness properties as follows.
\begin{defn} \label{B:dfn:Whitney}
Let 
\begin{equation*}\mathfrak{D}=\{\Tr_{m-1}\varphi:\varphi\text{ smooth and compactly supported in $\R^\dmn$}\}.\end{equation*}

We let $\dot W\!A^p_{m-1,0}(\R^n)$ be the completion of the set $\mathfrak{D}$ under the $L^p$ norm. 

We let $\dot W\!A^p_{m-1,1}(\R^n)$ be the completion of $\mathfrak{D}$ under the $\dot W^p_1(\R^n)$ norm, that is, under the norm $\doublebar{\arr f}_{\dot W\!A^p_{m-1,1}(\R^n)}=\doublebar{\nabla_\pureH \arr f}_{L^p(\R^n)}$. 

Finally, we let $\dot W\!A^2_{m-1,1/2}(\R^n)$ be the completion of $\mathfrak{D}$ under the norm 
\begin{equation}
\label{B:eqn:Besov:norm}
\doublebar{\arr f}_{\dot W\!A^2_{m-1,1/2}(\R^n)} = \biggl(\sum_{\abs\gamma=m-1}\int_{\R^n} \abs{\widehat {f_\gamma}(\xi)}^2\,\abs{\xi}\,d\xi\biggr)^{1/2}
\end{equation}
where $\widehat f$ denotes the Fourier transform of~$f$.
\end{defn}

We will see (Section~\ref{B:sec:Dirichlet}) that if $u$ is a solution to the differential equation \eqref{B:eqn:L} in $\R^\dmn_+$, and if $\mathcal{A}_2(t\nabla^m u)\in L^p(\R^n)$ or $\mathcal{A}_2(t\nabla^m \partial_t u)\in L^p(\R^n)$ for some $1<p\leq 2$, then up to a certain additive normalization, $\Tr_{m-1} u$ lies in $\dot W\!A^p_{m-1,0}(\R^n)$ or $\dot W\!A^p_{m-1,1}(\R^n)$. 

The space $\dot W\!A^2_{m-1,1/2}(\R^n)$ is of interest in connection with the theory of boundary value problems with solutions $u$ in $\dot W^2_m(\R^\dmn_+)$, as will be seen in the following lemma. Such boundary value problems may be investigated using the Lax-Milgram lemma, and many useful results may be obtained therefrom. In particular, we will define layer potentials (Section~\ref{B:sec:dfn:potentials}), establish duality results for layer potentials (Lemma~\ref{B:lem:dual:potentials}), and prove the Green's formula~\eqref{B:eqn:green}, in terms of such solutions.

\begin{lem}\label{B:lem:Besov} 
If $u\in \dot W^2_m(\R^\dmn_+)$ then $\Tr_{m-1}^+u\in \dot W\!A^2_{m-1,1/2}(\R^n)$, and furthermore
\begin{equation*}\doublebar{\Tr_{m-1}^+u}_{\dot W\!A^2_{m-1,1/2}(\R^n)}\leq C \doublebar{\nabla^m u}_{L^2(\R^\dmn_+)}.\end{equation*}
Conversely, if $\arr f\in \dot W\!A^2_{m-1,1/2}(\R^n)$, then there is some $F\in \dot W^2_m(\R^\dmn_+)$ such that $\Tr_{m-1}^+F=\arr f$ and such that
\begin{equation*}\doublebar{\nabla^m F}_{L^2(\R^\dmn_+)}\leq C \doublebar{\arr f}_{\dot W\!A^2_{m-1,1/2}(\R^n)}.\end{equation*}
\end{lem}

If $\dot W^2_m(\R^\dmn_+)$ and $\dot W\!A^2_{m-1,1/2}(\R^n)$ are replaced by their inhomogeneous counterparts, then this lemma is a special case of \cite{Liz60}. For the homogeneous spaces that we consider, the $m=1$ case of this lemma is a special case of \cite[Section~5]{Jaw77}. The trace result for $m\geq 2$ follows from the trace result for $m=1$;  extensions may easily be constructed using the Fourier transform.

\begin{rmk} This notion of Dirichlet boundary values may require some explanation. Most known results (see, for example, \cite{Ver90,PipV95B,MitM13A}) establish well-posedness of the Dirichlet problem for an elliptic differential operator of order $2m$ in the case where the Dirichlet boundary values of $u$ are taken to be $u\vert_{\partial\Omega}$, $\partial_\nu u\vert_{\partial\Omega}$, $\partial_\nu^2 u\vert_{\partial\Omega},\dots,\partial_\nu^{m-1} u\vert_{\partial\Omega}$, where $\partial_\nu$ denotes derivatives taken in the direction normal to the boundary. (Indeed the analogue to our Lemma~\ref{B:lem:Besov} in \cite{Liz60} is stated in this fashion.) 

If $\partial\Omega$ is connected, then up to adding polynomials, it is equivalent to specify the full gradient $\nabla^{m-1} u$ on the boundary. 
We prefer to specify $\Tr_{m-1}u=\Trace\nabla^{m-1}u$ rather than the array of functions $u\vert_{\partial\Omega}$, $\partial_\nu  u\vert_{\partial\Omega},\dots,\partial_\nu^{m-1} u\vert_{\partial\Omega}$ for reasons of homogeneousness. That is, we often expect all components of $\nabla^{m-1}u$ to exhibit the same degree of smoothness. This can be reflected by requiring all components of $\Tr_{m-1} u$ to lie in the same smoothness space, but the lower-order derivatives $u\vert_{\partial\Omega}$, $\partial_\nu u\vert_{\partial\Omega},\dots,\partial_\nu^{m-2} u\vert_{\partial\Omega}$ would have to lie in higher smoothness spaces. This is notationally awkward in $\R^\dmn_+$; furthermore, we hope in future to generalize to Lipschitz domains, in which case higher order smoothness spaces on the boundary are extremely problematic.
\end{rmk}

\subsubsection{Neumann boundary data}
\label{B:sec:dfn:Neumann}

It is by now standard to define Neumann boundary values in a variational sense. 

That is, suppose that  $u\in \dot W^2_m(\R^\dmn_+)$ and that $Lu=0$ in $\R^\dmn_+$. By the definition~\eqref{B:eqn:L} of~$Lu$, if $\varphi$ is smooth and supported in $\R^\dmn_+$, then $\langle \nabla^m\varphi,\mat A\nabla^m u \rangle_{\R^\dmn_+}=0$. By density of smooth functions and boundedness of the trace map, we have that $\langle \nabla^m\varphi,\mat A\nabla^m u \rangle_{\R^\dmn_+}=0$ for any $\varphi\in \dot W^2_m(\R^\dmn_+)$ with $\Tr_{m-1}^+\varphi=0$. Thus, if $\Psi\in \dot W^2_m(\R^\dmn_+)$, then $\langle \nabla^m\Psi,\mat A\nabla^m u\rangle_{\R^\dmn_+}$ depends only on $\Tr_{m-1}^+\Psi$. 

Thus, for solutions $u$ to $Lu=0$ with $u\in\dot W^2_m(\R^\dmn_+)$, we may define the Neumann boundary values $\M_{\mat A}^+ u$ by the formula
\begin{equation}\label{B:eqn:Neumann:W2}
\langle \Tr_{m-1}^+\Psi,\M_{\mat A}^+ u\rangle_{\R^n}
=
\langle \nabla^m\Psi,\mat A\nabla^m u\rangle_{\R^\dmn_+} 
\quad\text{for any $\Psi\in \dot W^2_m(\R^\dmn)$}.
\end{equation}
See \cite{BarM16B,BarHM15p} for a much more extensive discussion of higher order Neumann boundary values. 

We are interested in the Neumann boundary values of a solution $u$ to $Lu=0$ that satisfies $\mathcal{A}_2(t\nabla^m u)\in L^p(\R^n)$ or $\mathcal{A}_2(t\nabla^m\partial_t u)\in L^p(\R^n)$. For such functions the inner product \eqref{B:eqn:Neumann:W2} does not converge for arbitrary $\Psi\in \dot W^2_m(\R^\dmn_+)$.

If $\mathcal{A}_2(t\nabla^m u)\in L^2(\R^n)$, then $\nabla^m u$ is not even locally integrable near the boundary (see formula~\eqref{B:eqn:Lusin:2:norm}), and so the inner product~\eqref{B:eqn:Neumann:W2} will not in general converge even for smooth functions~$\Psi$ that are compactly supported in $\R^\dmn$. However, we will see (Section~\ref{B:sec:Neumann}) that for any $\arr \psi$ in the dense subspace $\mathfrak{D}$ of Definition~\ref{B:dfn:Whitney}, there is some extension $\Psi$ of $\arr \psi$ such that the inner product~\eqref{B:eqn:Neumann:W2} converges (albeit possibly not absolutely). We will thus define Neumann boundary values in terms of a distinguished extension.


Define the operator $\mathcal{Q}_t^m$ by
\begin{equation*}\mathcal{Q}_t^m = e^{-(-t^2\Delta_\pureH)^m}.\end{equation*}
Notice that if $f\in C^\infty_0(\R^n)$, then $\partial_t^k \mathcal{Q}_t^m f(x)\big\vert_{t=0}=0$ whenever $1\leq k\leq 2m-1$, and that $\mathcal{Q}_0^m f(x)=\varphi(x)$.

Suppose that $\varphi$ is smooth and compactly supported in $\R^\dmn$. Let $\varphi_k(x)=\partial_\dmn^k\varphi(x,0)$.
If $t\in\R$, let
\begin{equation}
\label{B:eqn:Neumann:extension}
\mathcal{E}\varphi(x,t) = \mathcal{E}(\Tr_{m-1}\varphi)(x,t) =  \sum_{k=0}^{m-1}\frac{1}{k!}t^k \mathcal{Q}^m_t \varphi_k(x).\end{equation}
Observe that $\mathcal{E}\varphi$ is also smooth on $\overline{\R^\dmn_+}$ up to the boundary, albeit is not compactly supported, and that $\Tr_{m-1}^+\mathcal{E}\varphi=\Tr_{m-1}^-\mathcal{E}\varphi= \Tr_{m-1}\varphi$. 

We define the Neumann boundary values $\M_{\mat A} u=\M_{\mat A}^+ u$ of $u$ by
\begin{equation}\label{B:eqn:Neumann:E}
\langle \M_{\mat A}^+ u,\Tr_{m-1}\varphi\rangle_{\R^n}
=
\lim_{\varepsilon\to 0^+} \lim_{T\to \infty}
\int_\varepsilon^T \langle \mat A\nabla^m u(\,\cdot\,,t), \nabla^m \mathcal{E}\varphi(\,\cdot\,,t)\rangle_{\R^n}\,dt
.\end{equation}
We define $\M_{\mat A}^- u$ similarly, as an appropriate integral from $-\infty$ to zero. Notice that $\M_{\mat A} u$ is an operator on the subspace $\mathfrak{D}$ appearing in Definition~\ref{B:dfn:Whitney}; given certain bounds on~$u$, we will prove boundedness results (see Section~\ref{B:sec:Neumann}) that allow us to extend $\M_{\mat A} u$ to an operator on $\dot W\!A^p_{m-1,0}(\R^n)$ or $\dot W\!A^p_{m-1,1}(\R^n)$ for various values of~$p$.

As mentioned in the introduction, if $\mathcal{A}_2(t\nabla^m u)\in L^p(\R^n)$ then the right-hand side of formula~\eqref{B:eqn:Neumann:W2} does represent an absolutely convergent integral even for $\Psi=\mathcal{E}\Tr_{m-1}^+\Psi$, and so the order of integration in formula~\eqref{B:eqn:Neumann:E} is important.

The two formulas \eqref{B:eqn:Neumann:W2} and~\eqref{B:eqn:Neumann:E} for the Neumann boundary values of a solution in $\dot W^2_m(\R^\dmn_+)$ coincide, as seen in the next lemma.

\begin{lem}\label{B:lem:Neumann:W2} Suppose that $\nabla^m u\in L^2(\R^\dmn_+)$ and that $Lu=0$ in $\R^\dmn_+$. Let $\varphi$ be smooth and compactly supported in $\R^\dmn$.

Then
\begin{equation*}\langle \nabla^m\varphi,\mat A\nabla^m u\rangle_{\R^\dmn_+} = \langle \nabla^m\mathcal{E}\varphi,\mat A\nabla^m u\rangle_{\R^\dmn_+}\end{equation*}
and so formulas~\eqref{B:eqn:Neumann:E} and \eqref{B:eqn:Neumann:W2} yield the same value for $\langle \Tr_{m-1}\varphi,\M_{\mat A}^+ u\rangle_{\R^n}$.

The operator $\M_{\mat A}^+ u$ as given by formula~\eqref{B:eqn:Neumann:W2} is a bounded operator on the space $\dot W\!A^2_{m-1,1/2}(\R^n)$, and  $\M_{\mat A}^+ u$ as given by formula~\eqref{B:eqn:Neumann:E} extends by density to the same operator on $\dot W\!A^2_{m-1,1/2}(\R^n)$.
\end{lem}

\begin{proof}
By an elementary argument involving the Fourier transform,
\begin{equation}\label{B:eqn:E:W2}
\doublebar{\nabla^m \mathcal{E}(\Tr_{m-1}\varphi)}_{L^2(\R^\dmn_\pm)} \leq C\doublebar{\Tr_{m-1}\varphi}_{\dot B^{2,2}_{1/2}(\R^n)}.\end{equation}
Thus, $\mathcal{E}\varphi$ is an extension of $\Tr_{m-1}\varphi$ in $\dot W^2_m(\R^\dmn_+)$, and so 
\begin{equation*}\langle \nabla^m\Psi,\mat A\nabla^m u\rangle_{\R^\dmn_+} = \langle \nabla^m\mathcal{E}\varphi,\mat A\nabla^m u\rangle_{\R^\dmn_+}\end{equation*} for any other extension $\Psi$ of $\Tr_{m-1}\varphi$ in $\dot W^2_m(\R^\dmn_+)$, in particular, for $\Psi=\varphi$. Boundedness of $\M_{\mat A}^+ u$ on $\dot W\!A^2_{m-1,1/2}(\R^n)$ follows from Lemma~\ref{B:lem:Besov}, and the lemma follows from density of the subspace $\mathfrak{D}$ of Definition~\ref{B:dfn:Whitney} in~$\dot W\!A^2_{m-1,1/2}(\R^n)$.
\end{proof}


\subsection{Potential operators}\label{B:sec:dfn:potentials}
Two very important tools in the theory of second order elliptic boundary value problems are the double and single layer potentials. These potential operators are also very useful in the higher order theory. 
In this section we define our formulations of higher-order layer potentials; this is the formulation used in \cite{BarHM15p,BarHM17pA} and is similar to that used in \cite{Agm57,CohG83,CohG85,Ver05,MitM13B,MitM13A}.

For any $\arr H\in L^2(\R^\dmn)$, by the Lax-Milgram lemma there is a unique function $u\in\dot W^2_m(\R^\dmn)$ that satisfies
\begin{equation}\label{B:eqn:newton}
\langle \nabla^m\varphi, \mat A\nabla^m u\rangle_{\R^\dmn}=\langle \nabla^m\varphi, \arr H\rangle_{\R^\dmn}\end{equation}
for all $\varphi\in \dot W^2_m(\R^\dmn)$.
Let $\Pi^L\arr H=u$.  We refer to $\Pi^L$ as the Newton potential operator for~$L$. See \cite{Bar16} for a further discussion of the operator~$\Pi^L$.

We will need the following duality relation (see \cite[Lemma~42]{Bar16}): if $\arr F\in L^2(\R^\dmn)$ and $\arr G\in L^2(\R^\dmn)$, then
\begin{align}\label{B:eqn:newton:adjoint}
\langle \arr F, \nabla^m\vec \Pi^L\arr G\rangle_{\R^\dmn} &= \langle \nabla^m\vec \Pi^{L^*}\arr F, \arr G\rangle_{\R^\dmn}
.\end{align}

We may define the double and single layer potentials in terms of the Newton potential.
Suppose that $\arr f\in \dot W\!A^2_{m-1,1/2}(\R^n)$.
By Lemma~\ref{B:lem:Besov}, there is some  $F\in \dot W^2_m(\R^\dmn_+)$ that satisfies $\arr f=\Tr_{m-1}^+ F$.
We define the double layer potential of $\arr f$ as
\begin{align}
\label{B:dfn:D:newton}
\D^{\mat A}\arr f &= -\1_+ F + \Pi^L(\1_+ \mat A\nabla^m F)
\end{align}
where $\1_+$ is the characteristic function of the upper half-space $\R^\dmn_+$.
$\D^{\mat A}\arr f$ is well-defined, that is, does not depend on the choice of~$F$; see \cite{BarHM15p}. We remark that by \cite[formula~(2.27)]{BarHM15p}, if $\1_-$ is the characteristic function of the lower half space, then
\begin{align}
\label{B:eqn:D:newton:lower}
\D^{\mat A}\arr f &= \1_- F - \Pi^L(\1_- \mat A\nabla^m F) \quad\text{if }\Tr_{m-1}^- F=\arr f.
\end{align}

Similarly, let $\arr g$ be a bounded operator on $\dot W\!A^2_{m-1,1/2}(\R^n)$. 
There is some $\arr G\in L^2(\R^\dmn_+)$ such that $\langle \arr G, \nabla^m\varphi\rangle_{\R^\dmn_+} = \langle \arr g, \Tr_{m-1}^+\varphi\rangle_{\partial{\R^\dmn_+}}$ for all $\varphi\in \dot W^2_m$; see \cite{BarHM15p}. Let $\1_+\arr G$ denote the extension of $\arr G$ by zero to $\R^\dmn$.
We define
\begin{align}
\label{B:dfn:S:newton}
\s^{L}\arr g&=\Pi^L(\1_+\arr G)
.\end{align}
Again, $\s^{L}\arr g$ does not depend on the choice of extension~$\arr G$.

It was shown in \cite{BarHM17pA} that the operators $\D^{\mat A}$ and $\s^L$, originally defined on $\dot W\!A^{2}_{m-1,1/2}(\R^n)$ and its dual space, extend by density to operators defined on $\dot W\!A^2_{m-1,0}(\R^n)$ and $\dot W\!A^2_{m-1,1}(\R^n)$ or their respective dual spaces; see Section~\ref{B:sec:potentials:bounds}.

%

A benefit of these formulations of layer potentials is the easy proof of the Green's formula. By taking $F=u$ and $\arr G=\mat A\nabla^m u$, we immediately have that
\begin{equation}
\label{B:eqn:green}
\1_+ \nabla^m u=-\nabla^m \D^{\mat A}(\Tr_{m-1}^+  u) + \nabla^m \s^{L}(\M_{\mat A}^+  u) 
\end{equation}
for all $u\in\dot W^2_m(\R^\dmn_+)$ that satisfy $Lu=0$ in $\R^\dmn_+$.

In the second-order case, a variant $\s^{L}\nabla$ of the single layer potential is often used; see \cite{AlfAAHK11,HofMitMor15,HofMayMou15}. We will define an analogous operator in this case.

Let $\alpha$ be a multiindex with $\abs\alpha=m$. If $\alpha_\dmn>0$, let 
\begin{equation}\label{B:eqn:S:S:vertical}\s^L_\nabla (h\arr e_\alpha)(x,t) = -\partial_t\s^L(h\arr e_\gamma)(x,t)\quad\text{where }\alpha=\gamma+\vec e_\dmn.\end{equation}
If $\alpha_\dmn<\abs\alpha=m$, then there is some $j$ with $1\leq j\leq n$ such that $\vec e_j\leq \alpha$. If $h$ is smooth and compactly supported, let 
\begin{equation}\label{B:eqn:S:S:horizontal}\s^L_\nabla (h\arr e_\alpha)(x,t) = -\s^L(\partial_{x_j} h\arr e_\gamma)(x,t)\quad\text{where }\alpha=\gamma+\vec e_j.\end{equation}
If $1\leq \alpha_\dmn\leq m-1$, then the two formulas \eqref{B:eqn:S:S:vertical} and \eqref{B:eqn:S:S:horizontal} coincide, and furthermore, the choice of distinguished direction $x_j$ in formula~\eqref{B:eqn:S:S:horizontal} does not matter; see \cite[formula~\eqref*{A:eqn:S:variant}]{BarHM17pA}. 

\section{Known results}
\label{B:sec:known}

To prove our main results, we will need to use a number of known results from the theory of higher order differential equations. We gather these results in this section.

\subsection{Regularity of solutions to elliptic equations}
\label{B:sec:Caccioppoli}

The first such result we list is the higher order analogue to the Caccioppoli inequality; it was proven in full generality in \cite{Bar16} and some important preliminary versions were established in \cite{Cam80,AusQ00}.
\begin{lem}[The Caccioppoli inequality]\label{B:lem:Caccioppoli}
Suppose that $L$ is a divergence-form elliptic operator associated to coefficients $\mat A$ satisfying the ellipticity conditions \eqref{B:eqn:elliptic} and~\eqref{B:eqn:elliptic:bounded}. Let $ u\in \dot W^2_m(B(X,2r))$ with $L u=0$ in $B(X,2r)$.

Then we have the bound
\begin{equation*}
\fint_{B(X,r)} \abs{\nabla^j  u(x,s)}^2\,dx\,ds
\leq \frac{C}{r^2}\fint_{B(X,2r)} \abs{\nabla^{j-1}  u(x,s)}^2\,dx\,ds
\end{equation*}
for any $j$ with $1\leq j\leq m$.
\end{lem}

Next, we mention the higher order generalization of Meyers's reverse H\"older inequality for gradients. The following theorem follows from the Caccioppoli inequality of \cite{Cam80,AusQ00,Bar16}, and was stated in some form in all three works. (The version given below comes most directly from \cite{Bar16}.)

\begin{thm}
\label{B:thm:Meyers} 
Let $L$ be an operator of order~$2m$ that satisfies the bounds \eqref{B:eqn:elliptic:bounded} and~\eqref{B:eqn:elliptic}. 
Then there is some number $p^+=p^+_0=p^+_L>2$ depending only on the standard constants
such that the following statement is true.

Let $X_0\in\R^\dmn$ and let $r>0$. Suppose that 
$L  u=0$ 
or $L^* u=0$ 
in $B(X_0,2r)$. 
Suppose that $0<p<q<p^+$. Then
\begin{align}
\label{B:eqn:Meyers}
\biggl(\int_{B(X_0,r)
}\abs{\nabla^m   u}^q\biggr)^{1/q}
&\leq 
	\frac{C(p,q)}{r^{\pdmn/p-\pdmn/q}}
	\biggl(\int_{B(X_0,2r)}\abs{\nabla^m   u}^p\biggr)^{1/p} 
\end{align}
for some constant $C(p,q)$ depending only on $p$, $q$ and the standard parameters.

We may also bound the lower-order derivatives. Let $1\leq k\leq m$. 
There is some extended real number $p^+_k$, with $p_k^+\geq p^+_L\,\pdmn/(\dmn-k\,p^+_L)$ if $\dmn>k\,p^+_L$ and with $p^+_k=\infty$ if $\dmn\leq k\,p^+_L$, such that the following is true.  Suppose $0<p< q<p^+_k$. Then 
\begin{align}\label{B:eqn:Meyers:lower}
\biggl(\int_{B(X_0,r)
}\abs{\nabla^{m-k}   u}^{q}\biggr)^{1/{q}}
&\leq 
	\frac{C(p,q)}{r^{\pdmn/p-\pdmn/{q}}}
	\biggl(\int_{B(X_0,2r)}\abs{\nabla^{m-k}   u}^{p}\biggr)^{1/p} 
.\end{align}
%
\end{thm}

We remark that if $\dmn=2$ then $p^+_1=\infty$. If $\dmn=3$ and $\mat A$ is $t$-independent, then again $p^+_1=\infty$; the argument presented in \cite[Appendix~B]{AlfAAHK11} in the case $m=1$ is valid in the higher order case. 

\subsection{Estimates on layer potentials}
\label{B:sec:potentials:bounds}

We will make extensive use of the following estimates on layer potentials from \cite{BarHM15p,BarHM17pA}, in particular the technical estimates \eqref{B:eqn:S:lusin:p:intro}, \eqref{B:eqn:S:lusin:variant:intro} and~\eqref{B:eqn:S:Schwartz:p}. (Indeed their applicability to this paper is the main reason the bounds \eqref{B:eqn:S:lusin:variant:intro} and~\eqref{B:eqn:S:Schwartz:p} were proven in \cite{BarHM17pA}.)

\begin{thm}\label{B:thm:square}\textup{(\cite[Theorem~1.1]{BarHM15p})}
Suppose that $L$ is an elliptic operator of the form \eqref{B:eqn:divergence} of order~$2m$, associated with coefficients $\mat A$ that are $t$-independent in the sense of formula~\eqref{B:eqn:t-independent} and satisfy the ellipticity conditions \eqref{B:eqn:elliptic} and~\eqref{B:eqn:elliptic:bounded}.

Then the operators $\D^{\mat A}$ and $\s^L$, originally defined on $\dot W\!A^{2}_{m-1,1/2}(\R^n)$ and its dual space, extend by density to operators that satisfy
\begin{align}
\label{B:eqn:S:square}
\int_{\R^n}\int_{-\infty}^\infty \abs{\nabla^m \partial_t\s^{L} \arr g(x,t)}^2\,\abs{t}\,dt\,dx
	& \leq C \doublebar{\arr g}_{L^2(\R^n)}^2
,\\
\label{B:eqn:D:square}
\int_{\R^n}\int_{-\infty}^\infty \abs{\nabla^m \partial_t  \D^{\mat A} \arr f(x,t)}^2\,\abs{t}\,dt\,dx
	& \leq C \doublebar{\arr f}_{\dot W^2_1(\R^\dmnMinusOne)}^2
	= C \doublebar{\nabla_\pureH\arr f}_{L^2(\R^n)}^2
\end{align}
for all $\arr g\in {L^2(\R^n)}$ and all $\arr f\in \dot W\!A^2_{m-1,1}(\R^n)$.
\end{thm}

\begin{thm}%
\label{B:thm:square:rough}%
\textup{(\cite[Theorems \ref*{A:thm:S:square:variant} and \ref*{A:thm:D:square:variant}]{BarHM17pA})}
Let $L$ be as in Theorem~\ref{B:thm:square}.
Then $\D^{\mat A}$ and $\s^L$ extend to operators that satisfy
\begin{align}
\label{B:eqn:S:square:variant}
\int_{\R^n}\int_{-\infty}^\infty \abs{\nabla^m \s_\nabla^{L} \arr h(x,t)}^2\,\abs{t}\,dt\,dx
	& \leq C \doublebar{\arr h}_{L^2(\R^n)}^2
,\\
\label{B:eqn:D:square:rough}
\int_{\R^n}\int_{-\infty}^\infty \abs{\nabla^m \D^{\mat A} \arr f(x,t)}^2\,\abs{t}\,dt\,dx
	& \leq C \doublebar{\arr f}_{L^2(\R^n)}^2
\end{align}
for all $\arr h\in L^2(\R^n)$ and all $\arr f\in \dot W\!A^2_{m-1,0}(\R^n)$.

\end{thm}

\begin{thm}\textup{(\cite[Theorem~\ref*{A:thm:Carleson:intro}]{BarHM17pA})}
\label{B:thm:potentials}
 Let $L$ be as in Theorem~\ref{B:thm:square}.
Suppose that $L$ is an elliptic operator associated with coefficients $\mat A$ that are $t$-independent in the sense of formula~\eqref{B:eqn:t-independent} and satisfy the ellipticity conditions \eqref{B:eqn:elliptic} and~\eqref{B:eqn:elliptic:bounded}.

If $k$ is large enough (depending on $m$ and~$n$), then the following statements are true.


There is some $\varepsilon>0$ such that we also have the area integral estimates
\begin{align}
\label{B:eqn:S:lusin:p:intro}
\doublebar{\mathcal{A}_2^\pm (\abs{t}^k\,\nabla^m \partial_t^{k} \s^{L} \arr g)}_{L^q(\R^n)}
&\leq C(k,q) \doublebar{\arr g}_{L^q(\R^n)}
,\\
\label{B:eqn:S:lusin:variant:intro}
\doublebar{\mathcal{A}_2^\pm (\abs{t}^{k+1}\,\nabla^m \partial_t^{k}\s^{L}_\nabla \arr h)}_{L^q(\R^n)}
&\leq C(k,q) \doublebar{\arr h}_{L^q(\R^n)}
\end{align}
for any $2-\varepsilon< q<\infty$. If $\dmn=2$ or $\dmn=3$ then the estimate \eqref{B:eqn:S:lusin:p:intro} is valid for $1<q<\infty$.

Finally, let $\eta$ be a Schwartz function defined on $\R^n$ with $\int\eta=1$. 
Let $\mathcal{Q}_t$ denote convolution with $\eta_t=t^{-n}\eta(\,\cdot\,,t)$. 
Let $\arr b$ be any array of bounded functions. Then for any $p$ with $1<p<\infty$, we have that
\begin{equation}
\label{B:eqn:S:Schwartz:p}
\doublebar{\mathcal{A}_2^\pm(\abs{t}^{k+1}\partial_\perp^{k+m}\s^{L}_\nabla (\arr b \mathcal{Q}_{\abs{t}} h))}_{L^p(\R^n)}
\leq C(p)\doublebar{\arr b}_{L^\infty(\R^n)}\doublebar{h}_{L^p(\R^n)}
\end{equation}
where the constant $C(p)$ depends only on $p$, $k$, the Schwartz constants of~$\eta$, and on the standard parameters $n$, $m$, $\lambda$, and~$\Lambda$.

\end{thm}

\section{Preliminaries}
\label{B:sec:preliminaries}

In this section we will prove some preliminary results that will be of use both in Section~\ref{B:sec:Dirichlet} (that is, to bound the Dirichlet traces of solutions) and in Section~\ref{B:sec:Neumann} (that is, to bound the Neumann traces of solutions).

\subsection{Regularity along horizontal slices}

In this section we will prove a regularity result for solutions to elliptic equations with $t$-independent coefficients. 
The following lemma was proven in the case $m=1$ in \cite[Proposition 2.1]{AlfAAHK11} and generalized to the case $m\geq 2$, $p=2$ in \cite[Lemma~3.2]{BarHM15p} and the case $m\geq 2$, $p$ arbitrary in \cite[Lemma~\ref*{A:lem:slices}]{BarHM17pA}.
\begin{lem}\label{B:lem:slices}
Let $t$ be a constant, and let $Q\subset\R^n$ be a cube.

Suppose that $\partial_s \arr u(x,s)$ satisfies the Caccioppoli-like inequality
\begin{equation*}
\biggl(\int_{B(X,r)} \abs{\partial_s \arr u(x,s)}^p\,dx\,ds\biggr)^{1/p}
\leq \frac{c_0}{r}\biggl(\int_{B(X,2r)} \abs{\arr u(x,s)}^p\,dx\,ds\biggr)^{1/p}
\end{equation*}
whenever $B(X,2r)\subset\{(x,s):x\in 2Q, t-\ell(Q)<s<t+\ell(Q)\}$, for some $1\leq p\leq\infty$. 

Then
\begin{equation*}\biggl(\int_Q \abs{\arr u(x,t)}^p\,dx\biggr)^{1/p} \leq C(p,c_0)\,\ell(Q)^{-1/p}
\biggl( 
\int_{2Q}\int_{t-\ell(Q)}^{t+\ell(Q)} \abs{\arr u(x,s)}^p\,ds\,dx\biggr)^{1/p}.\end{equation*}

In particular, if $Lu=0$ in $2Q\times(t-\ell(Q),t+\ell(Q))$, and $L$ is an operator of the form \eqref{B:eqn:divergence} of order~$2m$ associated to $t$-independent coefficients~$A$ that satisfy the ellipticity conditions \eqref{B:eqn:elliptic:bounded} and~\eqref{B:eqn:elliptic}, then 
\begin{equation*}\int_Q \abs{\nabla^{m-j} \partial_t^k u(x,t)}^p\,dx \leq \frac{C(p)}{\ell(Q)} 
\int_{2Q}\int_{t-\ell(Q)}^{t+\ell(Q)} \abs{\nabla^{m-j} \partial_s^k u(x,s)}^p\,ds\,dx\end{equation*}
for any $0\leq j\leq m$, any $0< p < p_j^+$, and any integer $k\geq 0$, where $p_j^+$ is as in Theorem~\ref{B:thm:Meyers}.
\end{lem}


\subsection{Duality results}



We will need the following duality results for layer potentials.

\begin{lem}\label{B:lem:dual:potentials}
Suppose that $L$ is an elliptic operator of order $2m$ associated with coefficients $\mat A$ that are $t$-independent in the sense of formula~\eqref{B:eqn:t-independent} and satisfy the ellipticity conditions \eqref{B:eqn:elliptic} and~\eqref{B:eqn:elliptic:bounded}.

Let $\arr f\in \dot W\!A^2_{m-1,1/2}(\R^n)$, let $\arr g$ lie in the dual space $(\dot W\!A^2_{m-1,1/2}(\R^n))^*$, and let $\arr \psi \in L^2(\R^n)$. Let $\tau>0$ and let $j\geq 0$ be an integer. Then 
\begin{align}\label{B:eqn:D:adjoint}
\langle \arr \psi, \nabla^{m}\partial_\tau ^j \D^{\mat A}\arr f(\,\cdot\,,\tau )\rangle_{\R^n}
&=
	(-1)^{j+1}\langle \M_{\mat A^*}^-(\partial_\dmn^j (\s^{L^*}_\nabla\arr \psi)_{-\tau }),\arr f\rangle_{\R^n}
,\\
\label{B:eqn:S:adjoint}
\langle \arr \psi ,\nabla^{m}\partial_\tau ^{j} 
	\s^{L} \arr g(\,\cdot\,,\tau)\rangle_{\R^n}
&=
	(-1)^j\langle\nabla^{m-1}  \partial_\dmn^j\s^{L^*}_\nabla\arr \psi(\,\cdot\,,-\tau) ,\arr g\rangle_{\R^n}
\end{align}
where $(\s^{L^*}_\nabla\arr \psi)_{-\tau }(x,s) = \s^{L^*}_\nabla\arr \psi(x,s-\tau)$.
\end{lem}

The proof will be based on the adjoint relation \eqref{B:eqn:newton:adjoint} for the Newton potential; we remark that the result may also be proven by writing layer potentials in terms of the fundamental solution (see \cite{BarHM15p,BarHM17pA}) and using the symmetry properties thereof.

\begin{proof}[Proof of Lemma~\ref{B:lem:dual:potentials}]
We begin with formula~\eqref{B:eqn:D:adjoint}.

Let $\arr q$ be smooth, compactly supported and integrate to zero.
By Lemma~\ref{B:lem:slices},
\begin{align*}
\langle \arr q, \nabla^{m-1}\partial_\tau ^j \D^{\mat A}\arr f(\,\cdot\,,\tau )\rangle_{\R^n}
&=
\partial_\tau ^j\langle \arr q, \nabla^{m-1} \D^{\mat A}\arr f(\,\cdot\,,\tau )\rangle_{\R^n}
.\end{align*}
Let $F\in \dot W^2_m(\R^\dmn_-)$ with $\Tr_{m-1}^-F=\arr f$; by Lemma~\ref{B:lem:Besov}, such an~$F$ must exist.
By formula \eqref{B:eqn:D:newton:lower} for the double layer potential,
\begin{align*}
\langle \arr q, \nabla^{m-1}\partial_\tau ^j \D^{\mat A}\arr f(\,\cdot\,,\tau )\rangle_{\R^n}
&=
	-\partial_\tau ^j \langle \arr q, \nabla^{m-1}\Pi^L(\1_-\mat A\nabla^m F)(\,\cdot\,,\tau )\rangle_{\R^n}
.\end{align*}
For the remainder of this proof, let subscripts denote translation in the vertical direction. That is, if $\varphi$ is a function (or array of functions) and $s\in\R$, let $\varphi_s(x,t)=\varphi(x,t+s)$. Notice that $\langle \varphi,\psi_s\rangle_{\R^\dmn}=\langle \varphi_{-s},\psi\rangle_{\R^\dmn}$. Then
\begin{align*}
\langle \arr q, \nabla^{m-1}\partial_\tau ^j \D^{\mat A}\arr f(\,\cdot\,,\tau )\rangle_{\R^n}
&=
	-\partial_\tau ^j \langle \arr q, \Tr_{m-1}^+(\Pi^L(\1_-\mat A\nabla^m F))_\tau \rangle_{\R^n}
\end{align*}

Recall the definition~\eqref{B:dfn:S:newton} of the single layer potential and let  $\arr Q$ be an array of functions supported in $\R^\dmn_+$ such that $\s^{L^*}\arr q=\Pi^{L^*}\arr Q$.
Then
\begin{align*}
\langle \arr q, \nabla^{m-1}\partial_\tau ^j \D^{\mat A}\arr f(\,\cdot\,,\tau )\rangle_{\R^n}
&=
	-\partial_\tau ^j \langle \1_+\arr Q, \nabla^{m}(\Pi^L(\1_-\mat A\nabla^m F))_\tau \rangle_{\R^\dmn}
\end{align*}
and by the adjoint relation \eqref{B:eqn:newton:adjoint},
\begin{align*}
\langle \arr q, \nabla^{m-1}\partial_\tau ^j \D^{\mat A}\arr f(\,\cdot\,,\tau )\rangle_{\R^n}
&=
	-\partial_\tau ^j \langle \mat A^*\nabla^{m}\Pi^{L^*}((\1_+\arr Q)_{-\tau }),\nabla^m F\rangle_{\R^\dmn_-}
.\end{align*}

Recall that if $\arr H\in L^2(\R^\dmn)$ then $u=\Pi^L \arr H$ is the \emph{unique} function in $\dot W^2_m(\R^\dmn)$ that satisfies formula~\eqref{B:eqn:newton}. 
If $\varphi\in \dot W^2_m(\R^\dmn)$,
then
\begin{equation*}\langle \nabla^m\varphi, \mat A^*\nabla^m (\Pi^{L^*}(\1_+\arr Q))_{-\tau}\rangle_{\R^\dmn}
=\langle \nabla^m\varphi_{\tau}, \mat A^*_{\tau}\nabla^m \Pi^{L^*}(\1_+\arr Q)\rangle_{\R^\dmn}.\end{equation*}
But if $\mat A$ is $t$-independent, then $\mat A^*=\mat A^*_{\tau}$, and so
\begin{align*}
\langle \nabla^m\varphi, \mat A^*\nabla^m (\Pi^{L^*}(\1_+\arr Q))_{-\tau}\rangle_{\R^\dmn}
&=\langle \nabla^m\varphi_{\tau}, \mat A^*\nabla^m \Pi^{L^*}(\1_+\arr Q)\rangle_{\R^\dmn}
\\&=\langle \nabla^m\varphi_{\tau}, \1_+\arr Q\rangle_{\R^\dmn}
=\langle \nabla^m\varphi, (\1_+\arr Q)_{-\tau}\rangle_{\R^\dmn}
.\end{align*}
Thus, $u=(\Pi^{L^*}(\1_+\arr Q))_{-\tau}$ satisfies formula~\eqref{B:eqn:newton} with $H=(\1_+\arr Q)_{-\tau}$, and so we must have
\begin{equation*}\nabla^m \Pi^{L^*}((\1_+\arr Q)_{-\tau}) = \nabla^m (\Pi^{L^*}(\1_+\arr Q))_{-\tau}=\nabla^m(\s^{L^*}\arr q)_{-\tau}\end{equation*}
as $L^2(\R^\dmn)$-functions.

Thus,
\begin{align*}
\langle \arr q, \nabla^{m-1}\partial_\tau ^j \D^{\mat A}\arr f(\,\cdot\,,\tau )\rangle_{\R^n}
&=
	(-1)^{j+1}\langle \mat A^*\nabla^{m} (\partial_\dmn^j\s^{L^*}\arr q)_{-\tau },\nabla^m F\rangle_{\R^\dmn_-}
.\end{align*}
By formulas~\eqref{B:eqn:S:S:vertical} and~\eqref{B:eqn:S:S:horizontal},
if $\arr \psi$ is smooth and compactly supported then
\begin{align*}
\langle \arr \psi, \nabla^{m}\partial_\tau ^j \D^{\mat A}\arr f(\,\cdot\,,\tau )\rangle_{\R^n}
&=
	(-1)^{j+1}\langle \mat A^*\nabla^{m}(\partial_\dmn^j \s^{L^*}_\nabla\arr \psi)_{-\tau },\nabla^m F\rangle_{\R^\dmn_-}
\end{align*}
By the bound \eqref{B:eqn:S:square:variant} and the Caccioppoli inequality, we may extend this relation to all $\arr \psi\in L^2(\R^n)$. Recalling the definition of Neumann boundary values, we have that
\begin{align*}
\langle \arr \psi, \nabla^{m}\partial_\tau ^j \D^{\mat A}\arr f(\,\cdot\,,\tau )\rangle_{\R^n}
&=
	(-1)^{j+1}\langle \M_{\mat A^*}^-(\partial_\dmn^j \s^{L^*}_\nabla\arr \psi)_{-\tau },\arr f\rangle_{\R^n}
\end{align*}
as desired.

We now turn to formula~\eqref{B:eqn:S:adjoint}. With $\arr q$ and $\arr Q$ as above, and with $\s^L\arr g=\Pi^L \arr G$,
\begin{align*}
\langle \arr q ,\nabla^{m-1}\partial_\tau ^{j} 
	\s^{L} \arr g(\,\cdot\,,\tau )\rangle_{\R^n}
&=
	\partial_\tau ^{j} \langle \arr q ,\nabla^{m-1}
	\Pi^L(\1_+\arr G)(\,\cdot\,,\tau )\rangle_{\R^n}
\\&=
	\partial_\tau ^{j} \langle (\1_+\arr Q)_{-\tau } ,\nabla^{m}
	\Pi^L(\1_+\arr G)\rangle_{\R^\dmn}
\end{align*}
and by formula~\eqref{B:eqn:newton:adjoint} as before,
\begin{align*}
\langle \arr q ,\nabla^{m-1}\partial_\tau ^{j} 
	\s^{L} \arr g(\,\cdot\,,\tau )\rangle_{\R^n}
&=
	\partial_\tau ^{j} \langle \nabla^m \Pi^{L^*}((\1_+\arr Q)_{-\tau }) ,\arr G\rangle_{\R^\dmn_+}
\\&=
	\partial_\tau ^{j} \langle \nabla^m ( \s^{L^*}\arr q)_{-\tau } ,\arr G\rangle_{\R^\dmn_+}
.\end{align*}
By definition of~$\arr G$, we have that
\begin{align*}
\langle \arr q ,\nabla^{m-1}\partial_\tau ^{j} 
	\s^{L} \arr g(\,\cdot\,,\tau )\rangle_{\R^n}
&=
	\partial_\tau ^{j} \langle \Tr_{m-1}^+ ( \s^{L^*}\arr q)_{-\tau } ,\arr g\rangle_{\R^n}
\\&=
	\partial_\tau ^{j} \langle\nabla^{m-1}  \s^{L^*}\arr q(\,\cdot\,,-\tau) ,\arr g\rangle_{\R^n}
.\end{align*}
Applying formulas~\eqref{B:eqn:S:S:horizontal} and~\eqref{B:eqn:S:S:vertical}, we see that
\begin{align*}
\langle \arr \psi ,\nabla^{m}\partial_\tau ^{j} 
	\s^{L} \arr g(\,\cdot\,,\tau )\rangle_{\R^n}
&=
	\partial_\tau^j\langle\nabla^{m-1}  \s^{L^*}_\nabla\arr \psi(\,\cdot\,,-\tau) ,\arr g\rangle_{\R^n}
\\&=
	(-1)^j\langle\nabla^{m-1}  \partial_\dmn^j\s^{L^*}_\nabla\arr \psi(\,\cdot\,,-\tau) ,\arr g\rangle_{\R^n}
\end{align*}
as desired.
\end{proof}

\subsection{Estimates in terms of area integral norms of solutions}
\label{B:sec:basic}



The main goal of this paper is to show that, if $Lu=0$ in $\R^\dmn_+$ and $u$ satisfies certain area integral estimates, then the Dirichlet and Neumann boundary values $\Tr_{m-1}^+ u$ and $\M_{\mat A}^+ u$ exist and are bounded.

Recall from formula~\eqref{B:eqn:Neumann:E} that $\M_{\mat A}^+ u$ is given by
\begin{equation*}\langle \arr \psi,\M_{\mat A}^+ u\rangle_{\R^n} = \int_0^\infty \langle \mat A^*\nabla^m \mathcal{E}\arr \psi(\,\cdot\,,s),\nabla^m u(\,\cdot\,,s)\rangle_{\R^n}\,ds.\end{equation*}
If $u$ decays fast enough, then we have the following formula for $\Tr_{m-1}^+ u$:
\begin{equation*}\langle \arr \psi, \Tr_{m-1}^+ u\rangle_{\R^n} 
= -\int_0^\infty \langle \arr \psi, \nabla^{m-1}\partial_s u(\,\cdot\,,s)\rangle_{\R^n}\,ds
= \int_0^\infty \langle -\mathcal{O}^+\arr \psi, \nabla^{m} u(\,\cdot\,,s)\rangle_{\R^n}\,ds\end{equation*}
for some constant matrix~$\mathcal{O}^+$. Thus, we wish to bound terms of the form
\begin{equation*}\int_0^\infty \langle \arr \psi_s,\nabla^m u(\,\cdot\,,s)\rangle_{\R^n}\,ds\end{equation*}
for some arrays $\arr \psi_s$.

We will prove the following technical lemma; passing from Lemma~\ref{B:lem:control:lusin} to our main results is the main work of Sections~\ref{B:sec:Dirichlet} and~\ref{B:sec:Neumann}.

\begin{lem}
\label{B:lem:control:lusin}
Suppose that $L$ is an elliptic operator of order $2m$ associated with coefficients $\mat A$ that are $t$-independent in the sense of formula~\eqref{B:eqn:t-independent} and satisfy the ellipticity conditions \eqref{B:eqn:elliptic} and~\eqref{B:eqn:elliptic:bounded}.

Suppose that $Lu=0$ in $\R^\dmn_+$. Suppose further that $\nabla^m u\in L^2(\R^n\times(\sigma,\infty))$ for any $\sigma>0$, albeit with $L^2$ norm that may approach $\infty$ as $\sigma \to 0^+$.

Let $j\geq m$ be an integer. Let $\omega$ be a nonnegative real-valued function, and for each $s>0$, let $\arr \psi_s \in L^2(\R^n)$. 
Then
\begin{multline*}
\int_0^\infty
	s^{2j}\omega(s) \abs{\langle \arr \psi_s,\nabla^{m}\partial_s^{2j} u(\,\cdot\,,s)\rangle_{\R^n}}\,ds
\\\leq
C_{j}
\int_{4/3}^{4}
	\int_{\R^n}
	\mathcal{A}_2^-(\abs{t}^{j-2m+1}
		\partial_\dmn^{j-m} \s^{L^*}_\nabla\arr \psi_{\abs{t} r } )(x)
	\,\mathcal{A}_2^+(
	\Omega(t)\, t\,
		{\nabla^m  u}
	)(x)
	\,dx
	\,dr 
\end{multline*}
where $\Omega(t)=\sup\{\omega(s):t\leq s\leq 4t\}$, provided the right-hand side is finite.
\end{lem}

\begin{proof}
Let $u_\tau(x,t)=u(x,t+\tau)$; by assumption, if $\tau>0$ then $\nabla^m u_\tau\in L^2(\R^\dmn_+)$. By the Caccioppoli inequality, if $\tau>0$ and $j\geq 0$ is an integer, then $\partial_\dmn^j u_\tau\in \dot W^2_m(\R^\dmn_+)$, and because $A$ is $t$-independent we have that $L(\partial_\dmn^j u_\tau)=0$ in $\R^\dmn_+$.

Let $s=2\tau$, so $u(x,s)=u_\tau(x,\tau)$. We will apply the Green's formula \eqref{B:eqn:green} to $\partial_\dmn^j u_\tau$.
Notice that by Lemma~\ref{B:lem:slices} and the Caccioppoli inequality, the map $\sigma \mapsto \nabla^m\partial_\dmn^j u_\tau(\,\cdot\,,\sigma )$ is continuous $(0,\infty)\mapsto L^2(\R^n)$. The Green's formula  is thus valid on horizontal slices $\R^n\times\{\tau \}$, and not only in~$\R^\dmn_+$.
Thus,
\begin{align*}
\langle \arr \psi_{2\tau},\nabla^{m}\partial_\dmn ^{2j} u_\tau(\,\cdot\,,\tau )\rangle_{\R^n}
&=
	-\langle \arr \psi_{2\tau} ,\nabla^{m}\partial_\dmn^{j} 
	\D^{\mat A} (\Tr_{m-1}^+ \partial_\dmn^j u_\tau)(\,\cdot\,,\tau )\rangle_{\R^n} 
	\\\nonumber&\qquad +
	\langle \arr \psi_{2\tau} ,\nabla^{m}\partial_\dmn ^{j} 
	\s^{L} (\M_{\mat A}^+ \partial_\dmn^j u_\tau)(\,\cdot\,,\tau )\rangle_{\R^n}
.\end{align*}

\begin{rmk} This application of the Green's formula is the only time in the proof of this lemma that we use the fact that $u_\sigma\in \dot W^2_m(\R^\dmn_+)$. We will also assume $v_\sigma \in \dot W^2_m(\R^\dmn_+)$ and $w_\sigma \in \dot W^2_m(\R^\dmn_+)$ in Theorems~\ref{B:thm:Dirichlet:1}, \ref{B:thm:Dirichlet:2}, \ref{B:thm:Neumann:1} and~\ref{B:thm:Neumann:2}; again, that assumption is necessary only in order to apply Lemma~\ref{B:lem:control:lusin}, and so only necessary to ensure validity of the Green's formula.
\end{rmk}

By Lemma~\ref{B:lem:dual:potentials}, we have that
\begin{align*}
\langle \arr \psi_{2\tau},\nabla^{m}\partial_\dmn ^{2j} u_\tau(\,\cdot\,,\tau )\rangle_{\R^n}
&=	
	(-1)^{j}\langle \M_{\mat A^*}^-(\partial_\dmn^j (\s^{L^*}_\nabla\arr \psi_{2\tau})_{-\tau }),\Tr_{m-1}^+ \partial_\dmn^j u_\tau\rangle_{\R^n}		
	\\\nonumber&\qquad +
	(-1)^j\langle\nabla^{m-1}  \partial_\dmn^j\s^{L^*}_\nabla\arr \psi_{2\tau}(\,\cdot\,,-\tau) ,\M_{\mat A}^+ \partial_\dmn^j u_\tau\rangle_{\R^n}
.\end{align*}
Recall the definition of the Neumann boundary operator for $\dot W^2_m(\R^\dmn_\pm)$-functions. Let $0<\varepsilon\ll 1$ be a small fixed absolute constant, to be chosen later. Let $\eta_\tau(r)=\eta(r/\varepsilon\tau)$, where $\eta:\R\mapsto\R$ is a smooth function with $\abs{\eta(r)}=1$ if $\abs{r}<1/2$ and $\abs{\eta(r)}=0$ if $\abs{r}>1$. Recalling the definition of $u_\tau$ and $(\partial_\dmn^j (\s^{L^*}_\nabla\arr \psi_{2\tau})_{-\tau })$, we have that
\begin{multline*}
(-1)^j\langle \arr \psi_{2\tau},\nabla^{m}\partial_\dmn^{2j} u_\tau(\,\cdot\,,\tau )\rangle_{\R^n}
\\\begin{aligned}
	&=
	\int_{-\varepsilon\tau}^0
	\int_{\R^n}
	\langle \mat A^*(z)\nabla^{m}\partial_r^j \s^{L^*}_\nabla\arr \psi_{2\tau}(z,r-\tau ),\nabla^m (\eta_\tau(r)\,\partial_r^j u(z,r+\tau))\rangle \,dz\,dr
	\\\nonumber&\qquad +
	\int_0^{\varepsilon\tau}
	\int_{\R^n}
	\langle \nabla^m (\eta_\tau(r)\partial_r^{j} \s^{L^*}_\nabla\arr \psi_{2\tau}(z,r-\tau )) ,\mat A(z)\nabla^m \partial_r^j u(z,r+\tau)\rangle\,dz\,dr
.\end{aligned}
\end{multline*}
Because $\eta$ is smooth, we have that $\abs{\partial_r^k\eta_\tau(r)}\leq C_{k,\varepsilon} \tau^{-k}$, and so if $j\geq m$, then
{\multlinegap=0pt \begin{multline*}
\abs{\langle \arr \psi_{2\tau},\nabla^{m}\partial_\dmn^{2j} u_\tau(\,\cdot\,,\tau )\rangle_{\R^n}}
\\\begin{aligned}
	&\leq C_{j,\varepsilon}\sum_{k=j-m}^j
	\int_{-\varepsilon\tau}^0
	\int_{\R^n}
	\abs{\nabla^{m}\partial_r^j \s^{L^*}_\nabla\arr \psi_{2\tau}(z,r-\tau )} \, \tau^{k-j}\abs{\nabla^m \partial_r^k u(z,r+\tau)} \,dz\,dr
	\\\nonumber&\qquad +
	C_{j,\varepsilon}\sum_{\ell=j-m}^j
	\int_0^{\varepsilon\tau}
	\int_{\R^n}
	\tau^{\ell-j}\abs{\nabla^m \partial_r^{\ell} \s^{L^*}_\nabla\arr \psi_{2\tau}(z,r-\tau ) }\, \abs{\nabla^m \partial_r^k u(z,r+\tau)}\,dz\,dr
.\end{aligned}
\end{multline*}}%
Thus
\begin{multline*}
\int_0^\infty
	s^{2j}\omega(s) \abs{\langle \arr \psi_s,\nabla^{m}\partial_s^{2j} u(\,\cdot\,,s)\rangle_{\R^n}}\,ds
\\\leq
C_{j,\varepsilon}\int_0^\infty
	\omega(2\tau) 
	\sum_{k,\ell}
	\int_{-\varepsilon\tau}^{\varepsilon\tau}
	\int_{\R^n}
	\tau^{\ell+k}\abs{\nabla^m \partial_r^{\ell} \s^{L^*}_\nabla\arr \psi_{2\tau}(z,r-\tau ) }\\\times \abs{\nabla^m \partial_r^k u(z,r+\tau)}\,dz\,dr
	\,d\tau
.\end{multline*}
Making the change of variables $r=\theta\tau$, we have that
\begin{multline*}
\int_0^\infty
	s^{2j}\omega(s) \abs{\langle \arr \psi_s,\nabla^{m}\partial_s^{2j} u(\,\cdot\,,s)\rangle_{\R^n}}\,ds
\\\leq
C_{j,\varepsilon}\int_0^\infty
	\omega(2\tau) 
	\sum_{k,\ell}
	\int_{-\varepsilon}^{\varepsilon}
	\int_{\R^n}
	\tau^{\ell+k+1}\abs{\nabla^m \partial_\dmn^{\ell} \s^{L^*}_\nabla\arr \psi_{2\tau}(z,(\theta-1)\tau ) }\\\times \abs{\nabla^m \partial_\dmn^k u(z,(1+\theta)\tau)}\,dz\,d\theta
	\,d\tau
\end{multline*}
and changing the order of integration we see that
\begin{multline*}
\int_0^\infty
	s^{2j}\omega(s) \abs{\langle \arr \psi_s,\nabla^{m}\partial_s^{2j} u(\,\cdot\,,s)\rangle_{\R^n}}\,ds
\\\leq
C_{j,\varepsilon}
	\sum_{k,\ell}
	\int_{-\varepsilon}^{\varepsilon}
	\int_0^\infty
	\int_{\R^n}
	\tau^{\ell+k+1}\abs{\nabla^m \partial_\dmn^{\ell} \s^{L^*}_\nabla\arr \psi_{2\tau}(z,(\theta-1)\tau ) }
	\\\times 
	\omega(2\tau) 
	\abs{\nabla^m \partial_\dmn^k u(z,(1+\theta)\tau)}\,dz
	\,d\tau\,d\theta
.\end{multline*}
Now, observe that if $F$ is a nonnegative function and $a>0$ then
\begin{equation}\label{B:eqn:area:0}
\int_{\R^n}\int_0^\infty F(z,\tau)\,d\tau\,dz = 
\frac{C_n}{a^n}
\int_{\R^n}\int_0^\infty \int_{\abs{z-x}<a\tau} F(z,\tau)\frac{1}{\tau^n}\,dz\,d\tau\,dx.\end{equation}
Thus,
\begin{multline*}
\int_0^\infty
	s^{2j}\omega(s) \abs{\langle \arr \psi_s,\nabla^{m}\partial_s^{2j} u(\,\cdot\,,s)\rangle_{\R^n}}\,ds
\\\leq
C_{j,\varepsilon}
	\sum_{k,\ell}
	\int_{-\varepsilon}^{\varepsilon}
	\int_{\R^n}
	\int_0^\infty
	\int_{\abs{z-x}<\varepsilon\tau}
	\tau^{\ell+k+1-n}\abs{\nabla^m \partial_\dmn^{\ell} \s^{L^*}_\nabla\arr \psi_{2\tau}(z,(\theta-1)\tau ) }
	\\\times 
	\omega(2\tau) 
	\abs{\nabla^m \partial_\dmn^k u(z,(1+\theta)\tau)}\,dz
	\,d\tau\,dx\,d\theta
.\end{multline*}
By H\"older's inequality,
\begin{multline*}
\int_0^\infty
	s^{2j}\omega(s) \abs{\langle \arr \psi_s,\nabla^{m}\partial_s^{2j} u(\,\cdot\,,s)\rangle_{\R^n}}\,ds
\\\leq
C_{j,\varepsilon}
	\sum_{k,\ell}
	\int_{-\varepsilon}^{\varepsilon}
	\int_{\R^n}
	\int_0^\infty
	\biggl(\int_{\abs{z-x}<\varepsilon\tau}
		\abs{\nabla^m \partial_\dmn^{\ell} \s^{L^*}_\nabla\arr \psi_{2\tau}(z,(\theta-1)\tau ) }^2
	\,dz\biggr)^{1/2}
	\\\times 
	\biggl(\int_{\abs{z-x}<\varepsilon\tau}
		\abs{\nabla^m \partial_\dmn^k u(z,(1+\theta)\tau)}^2
	\,dz\biggr)^{1/2}
	\tau^{\ell+k+1-n}\omega(2\tau) 
	\,d\tau\,dx\,d\theta
.\end{multline*}
By Lemma~\ref{B:lem:slices}, and recalling that $\abs\theta\leq \varepsilon$, we have that
\begin{equation*}
\int_{\abs{z-x}<\varepsilon\tau}
		\abs{\nabla^m \partial_\dmn^k u(z,(1+\theta)\tau)}^2
	\,dz
	\leq
	\frac{C_\varepsilon}{\tau}
\int_{(1-2\varepsilon)\tau}^{(1+2\varepsilon)\tau}
\int_{\abs{z-x}<2\varepsilon\tau}\!
		\abs{\nabla^m \partial_r^k u(z,r)}^2
	\,dz\,dr
.\end{equation*}
By the Caccioppoli inequality,
\begin{equation*}
\int_{\abs{z-x}<\varepsilon\tau}
		\abs{\nabla^m \partial_\dmn^k u(z,(1+\theta)\tau)}^2
	\,dz
	\leq
	\frac{C_\varepsilon}{\tau^{1+2k}}
\int_{(1-3\varepsilon)\tau}^{(1+3\varepsilon)\tau}
\int_{\abs{z-x}<3\varepsilon\tau}\!\!
		\abs{\nabla^m  u(z,r)}^2
	\,dz\,dr
.\end{equation*}
By Theorem~\ref{B:thm:Meyers}, we have that
\begin{multline*}
\biggl(\int_{\abs{z-x}<\varepsilon\tau}
		\abs{\nabla^m \partial_\dmn^k u(z,(1+\theta)\tau)}^2
	\,dz\biggr)^{1/2}
\\\leq
	\frac{C_\varepsilon}{\tau^{k+n/2+1}}
\int_{(1-4\varepsilon)\tau}^{(1+4\varepsilon)\tau}
\int_{\abs{z-x}<4\varepsilon\tau}
		\abs{\nabla^m  u(z,r)}
	\,dz\,dr
.\end{multline*}
Letting $r=\pi \tau$, we have that
\begin{multline*}
\biggl(\int_{\abs{z-x}<\varepsilon\tau}
		\abs{\nabla^m \partial_\dmn^k u(z,(1+\theta)\tau)}^2
	\,dz\biggr)^{1/2}
\\\leq
	\frac{C_\varepsilon}{\tau^{k+n/2}}
\int_{1-4\varepsilon}^{1+4\varepsilon}
\int_{\abs{z-x}<4\varepsilon\tau}
		\abs{\nabla^m  u(z,\pi\tau)}
	\,dz\,d\pi
.\end{multline*}
By an identical argument,
%
\begin{multline*}
\biggl(\int_{\abs{z-x}<\varepsilon\tau}
		\abs{\nabla^m \partial_\dmn^{\ell} \s^{L^*}_\nabla\arr \psi_{2\tau}(z,(\theta-1)\tau ) }^2
	\,dz\biggr)^{1/2}
\\\leq
	\frac{C_\varepsilon}{\tau^{2m+\ell-j+n/2}}
\int_{-1-4\varepsilon}^{-1+4\varepsilon}
\int_{\abs{z-x}<4\varepsilon\tau}
		\abs{\partial_\dmn^{j-m} \s^{L^*}_\nabla\arr \psi_{2\tau}(z,\kappa\tau ) }
	\,dz
	\,d\kappa
.\end{multline*}

Thus,
\begin{multline*}
\int_0^\infty
	s^{2j}\omega(s) \abs{\langle \arr \psi_s,\nabla^{m}\partial_s^{2j} u(\,\cdot\,,s)\rangle_{\R^n}}\,ds
\\\leq
C_{j,\varepsilon}
\int_{1-4\varepsilon}^{1+4\varepsilon}
\int_{-1-4\varepsilon}^{-1+4\varepsilon}
	\int_{\R^n}
	\int_0^\infty
	\int_{\abs{z-x}<4\varepsilon\tau}
		\abs{\nabla^m  u(z,\pi\tau)}\,\omega(2\tau) 
	\,dz
	\\\times 
	\tau^{1+j-2n-2m}
	\int_{\abs{z-x}<4\varepsilon\tau}
		\abs{\partial_\dmn^{j-m} \s^{L^*}_\nabla\arr \psi_{2\tau}(z,\kappa\tau ) }
	\,dz
	\,d\tau\,dx
	\,d\kappa\,d\pi
.\end{multline*}
Applying H\"older's inequality, we see that
\begin{multline*}
\int_0^\infty
	s^{2j}\omega(s) \abs{\langle \arr \psi_s,\nabla^{m}\partial_s^{2j} u(\,\cdot\,,s)\rangle_{\R^n}}\,ds
\\\leq
C_{j,\varepsilon}
\int_{1-4\varepsilon}^{1+4\varepsilon}
\int_{-1-4\varepsilon}^{-1+4\varepsilon}
	\int_{\R^n}
	\biggl(\int_0^\infty
	\int_{\abs{z-x}<4\varepsilon\tau}
	\omega(2\tau)^2\tau^{1-n}
		\abs{\nabla^m  u(z,\pi\tau)}^2
	\,dz
	\,d\tau\biggr)^{1/2}
	\\\times 
	\biggl(\int_0^\infty
	\int_{\abs{z-x}<4\varepsilon\tau}
		\abs{\partial_\dmn^{j-m} \s^{L^*}_\nabla\arr \psi_{2\tau}(z,\kappa\tau ) }^2\tau^{2j-4m-n+1}
	\,dz\,d\tau\biggr)^{1/2}
	\,dx
	\,d\kappa\,d\pi
.\end{multline*}
Apply the change of variables $t=\pi\tau$ in the first integral and $t=\kappa\tau$ in the second integral. We then have that
\begin{multline*}
\int_0^\infty
	s^{2j}\omega(s) \abs{\langle \arr \psi_s,\nabla^{m}\partial_s^{2j} u(\,\cdot\,,s)\rangle_{\R^n}}\,ds
\\\leq
C_{j,\varepsilon}
\int_{1-4\varepsilon}^{1+4\varepsilon}
\int_{-1-4\varepsilon}^{-1+4\varepsilon}
	\int_{\R^n}
	\biggl(\int_0^\infty
	\int_{\abs{z-x}<4\varepsilon t/\pi}\!\!\!
	\omega(2t/\pi)^2 t^{1-n}
		\abs{\nabla^m  u(z,t)}^2
	\,dz
	\,dt\biggr)^{1/2}
	\\\times 
	\biggl(\int_{-\infty}^0
	\int_{\abs{z-x}<4\varepsilon  t/\kappa}
		\abs{\partial_\dmn^{j-m} \s^{L^*}_\nabla\arr \psi_{2t/\kappa}(z,t ) }^2 \abs{t}^{2j-4m+1-n}
	\,dz\,dt\biggr)^{1/2}
	\,dx
	\,d\kappa\,d\pi
.\end{multline*}
Let $\varepsilon=1/8$. Because $\pi\geq 1-4\varepsilon=1/2$, we have that $4\varepsilon/\pi\leq 1$. Similarly, $4\varepsilon/\abs\kappa\leq 1$. Recall $\Omega(t)=\sup\{\omega(s):t\leq s\leq 4t\}$; then $\omega(2t/\pi)\leq \Omega(t)$. So
\begin{multline*}
\int_0^\infty
	s^{2j}\omega(s) \abs{\langle \arr \psi_s,\nabla^{m}\partial_s^{2j} u(\,\cdot\,,s)\rangle_{\R^n}}\,ds
\\\leq
C_{j}
\int_{-3/2}^{-1/2}
	\int_{\R^n}
	\biggl(\int_0^\infty
	\int_{\abs{z-x}<t}
	\Omega(t)^2 t^{1-n}
		\abs{\nabla^m  u(z,t)}^2
	\,dz
	\,dt\biggr)^{1/2}
	\\\times 
	\biggl(\int_{-\infty}^0
	\int_{\abs{z-x}<\abs{t}}
		\abs{\partial_\dmn^{j-m} \s^{L^*}_\nabla\arr \psi_{2t/\kappa}(z,t ) }^2 \abs{t}^{2j-4m+1-n}
	\,dz\,dt\biggr)^{1/2}
	\,dx
	\,d\kappa
.\end{multline*}
Recalling the definition of $\mathcal{A}_2$, we see that
\begin{multline*}
\int_0^\infty
	s^{2j}\omega(s) \abs{\langle \arr \psi_s,\nabla^{m}\partial_s^{2j} u(\,\cdot\,,s)\rangle_{\R^n}}\,ds
\\\leq
C_{j}
\int_{-3/2}^{-1/2}
	\int_{\R^n}
	\mathcal{A}_2^-(\abs{t}^{j-2m+1}
		\partial_\dmn^{j-m} \s^{L^*}_\nabla\arr \psi_{2{t}/\kappa} )(x)
	\,\mathcal{A}_2^+(
	\Omega(t)\, t\,
		{\nabla^m  u}
	)(x)
	\,dx
	\,d\kappa
.\end{multline*}
Making the change of variables $r=-2/\kappa$ completes the proof.
\end{proof}

\section{The Dirichlet boundary values of a solution}
\label{B:sec:Dirichlet}

In this section we will prove results pertaining to the Dirichlet boundary values. Specifically, we will prove the following two theorems.

\begin{thm} 
\label{B:thm:Dirichlet:1}
Suppose that $L$ is an elliptic operator associated with coefficients $\mat A$ that are $t$-independent in the sense of formula~\eqref{B:eqn:t-independent} and satisfy the ellipticity conditions \eqref{B:eqn:elliptic} and~\eqref{B:eqn:elliptic:bounded}.
Let $v\in \dot W^2_{m,loc}(\R^\dmn_+)$ and suppose that $Lv=0$ in $\R^\dmn_+$.

Suppose that
$\doublebar{\mathcal{A}_2^+  (t\,\nabla^m v)}_{L^{p}(\R^n)} < \infty$
for some $p$ with $1<p<p_1^+$, where $p_1^+$ is as in Theorem~\ref{B:thm:Meyers}, and where for some $k\geq 1$ the bound
\begin{align}
\label{B:eqn:S:lusin:p}
\doublebar{\mathcal{A}_2^- (t^k\,\partial_t^{m+k} \s^{L^*} \arr g)}_{L^{p'}(\R^n)}&\leq C(p') \doublebar{\arr g}_{L^{p'}(\R^n)}
\end{align}
is valid for all $\arr g\in L^{p'}(\R^n)$. Here $1/p+1/p'=1$. Suppose in addition that, for all $\sigma>0$, we have that $\nabla^m v\in L^2(\R^n\times(\sigma,\infty))$, albeit possibly with a norm that approaches $\infty$ as $\sigma\to 0^+$.

Then there is some function $P$ defined in $\R^\dmn_+$ with $\nabla^m P=0$ (that is, a polynomial of degree at most $m-1$) such that
\begin{align*}
\sup_{t>0} \doublebar{\nabla^{m-1} v(\,\cdot\,,t)-\nabla^{m-1} P}_{L^p(\R^n)}
&\leq 
C\doublebar{\mathcal{A}_2^+ (t\,\nabla^m v)}_{L^{p}(\R^n)}
,
\\
\lim_{t\to \infty} \doublebar{\nabla^{m-1} v(\,\cdot\,,t)-\nabla^{m-1} P}_{L^p(\R^n)}&=0
.\end{align*}
Furthermore, there is some array of functions $\arr f\in L^1_{loc}(\R^n)$ such that 
\begin{equation*}\doublebar{\nabla^{m-1} v(\,\cdot\,,t)- \arr f}_{L^p(\R^n)}\to 0\quad\text{ as $t\to 0^+$},\end{equation*}
and such that
\begin{equation*}
\doublebar{\arr f-\nabla^{m-1} P}_{L^p(\R^n)}\leq C\doublebar{\mathcal{A}_2^+ (t\,\nabla^m v)}_{L^{p}(\R^n)}
.\end{equation*}
\end{thm}

\begin{thm} 
\label{B:thm:Dirichlet:2}
Suppose that $L$ is an elliptic operator associated with coefficients $\mat A$ that are $t$-independent in the sense of formula~\eqref{B:eqn:t-independent} and satisfy the ellipticity conditions \eqref{B:eqn:elliptic} and~\eqref{B:eqn:elliptic:bounded}.
Let $w\in \dot W^2_{m,loc}(\R^\dmn_+)$ and suppose that $Lw=0$ in $\R^\dmn_+$.

Suppose that
$\doublebar{\mathcal{A}_2^+  (t\,\nabla^m \partial_t w)}_{L^{p}(\R^n)} < \infty$
for some $p$ with $1<p<p_0^+$, where $p_0^+=p_L^+$ is as in Theorem~\ref{B:thm:Meyers}, and where for some $k\geq 1$ the bound
\begin{align}
\label{B:eqn:S:lusin:variant}
\doublebar{\mathcal{A}_2^- (t^{k}\,\partial_t^{m+k-1} \s^{L^*}_\nabla \arr h)}_{L^{p'}(\R^n)}&\leq C(p') \doublebar{\arr h}_{L^{p'}(\R^n)}
\end{align}
is valid for all $\arr h\in L^{p'}(\R^n)$. Suppose in addition that $\nabla^m \partial_\dmn w\in L^2(\R^n\times(\sigma,\infty))$ for all $\sigma>0$.

Then there is some array $\arr p$ of functions defined on $\R^n$ such that
\begin{align*}
\sup_{t>0} \doublebar{\nabla^{m} w(\,\cdot\,,t)-\arr p}_{L^p(\R^n)}
&\leq 
C\doublebar{\mathcal{A}_2^+ (t\,\nabla^m \partial_t w)}_{L^{p}(\R^n)}
,\\
\lim_{t\to\infty} \doublebar{\nabla^{m} w(\,\cdot\,,t)-\arr p}_{L^p(\R^n)}
&=0
.\end{align*}
Furthermore, there is some array of functions $\arr f\in L^1_{loc}(\R^n)$ such that 
\begin{equation*}\doublebar{\nabla^{m} w(\,\cdot\,,t)- \arr f}_{L^p(\R^n)}\to 0\quad\text{ as $t\to 0^+$},\end{equation*}
and such that
\begin{equation*}
\doublebar{\arr f-\arr p}_{L^p(\R^n)}\leq C\doublebar{\mathcal{A}_2^+ (t\,\nabla^m \partial_t w)}_{L^{p}(\R^n)}
.\end{equation*}

If $\nabla^{m} w(\,\cdot\,,t)\in L^p(\R^n)$ for some $t>0$, then $\arr p=0$. Otherwise, the array $\arr p$ satisfies $\arr p(x)=\nabla^m P(x,t)$, for some function $P\in \dot W^2_{m,loc}(\R^\dmn_+)$ such that
\begin{itemize}
\item $P(x,t)=P_1(x,t)+P_2(x)$,
\item $P_1(x,t)$ is a polynomial of degree at most~$m$ (and so $\nabla^m P_1$ is constant),
\item $P_2\in \dot W^2_{m,loc}(\R^n)$,
\item $LP=0$ and so 
\begin{equation}
\label{B:eqn:LP}
\sum_{\substack{\abs\alpha=\abs\beta=m\\\alpha_\dmn=\beta_\dmn=0}} \partial_\pureH^\alpha(A_{\alpha\beta}(x)\partial_\pureH^\beta P_2(x)) = -\sum_{\substack{\abs\alpha=\abs\beta=m\\\alpha_\dmn=0}} \partial_\pureH^\alpha(A_{\alpha\beta}(x) \partial^\beta P_1).\end{equation}
\end{itemize}
\end{thm}

\begin{rmk}\label{B:rmk:Dirichlet:W2} 
If $W_{p,q}$ is as in formula~\eqref{B:eqn:w:bound}, then by Theorem~\ref{B:thm:Meyers} and Lemma~\ref{B:lem:slices}, we have that 
\[\doublebar{\nabla^{m} w(\,\cdot\,,t)}_{L^p(\R^n)}\leq C_{p,q} W_{p,q}(t)\]
and so as in Theorem~\ref{B:thm:intro}, finiteness of $W_{p,q}(t)$ suffices to imply $w=w-P$.

Recall that if $1<p<2+\varepsilon$, then by Theorem~\ref{B:thm:potentials} the bounds \eqref{B:eqn:S:lusin:p} and~\eqref{B:eqn:S:lusin:variant} are valid.

Furthermore, we claim that if $\mathcal{A}_2^+  (t\,\nabla^m v)\in L^p(\R^n)$ or $\mathcal{A}_2^+  (t\,\nabla^m \partial_t w)\in L^p(\R^n)$ for some $p\leq 2$, then $\nabla^m v\in L^2(\R^n\times(\sigma,\infty))$ or $\nabla^m \partial_\dmn w\in L^2(\R^n\times(\sigma,\infty))$ for all $\sigma>0$.

To verify this, let $u=v$ or $u=\partial_\dmn w$. Let $c\geq 1$ and let $K$ be a large integer such that $c2^{-K}<\sigma$. Then
\begin{equation*}\int_{\R^n}\int_{\sigma}^\infty\abs{\nabla^m u(x,t)}^2\,dt\,dx 
\leq
\sum_{j=-K}^\infty \sum_{Q\in \mathcal{G}_j} 
\int_Q \int_{c\ell(Q)}^{2c\ell(Q)}\abs{\nabla^m u(x,t)}^2\,dt\,dx
\end{equation*} 
where $\mathcal{G}_j$ is a grid of cubes in $\R^n$ of side-length $2^j$. But if $c$ is large enough, then for any $y\in Q$,
\begin{equation*}\int_Q \int_{c\ell(Q)}^{2c\ell(Q)}\abs{\nabla^m u(x,t)}^2\,dt\,dx
\leq C\ell(Q)^{n-1}\int_0^\infty \int_{\abs{x-y}<s} \abs{\nabla^m u(x,t)}^2\frac{1}{t^{n-1}}\,dt\,dx\end{equation*}
and so
\begin{align*}\int_{\R^n}\int_{\sigma}^\infty\abs{\nabla^m u(x,t)}^2\,dt\,dx 
&\leq
\sum_{j=-K}^\infty 
\sum_{Q\in \mathcal{G}_j} 
C\ell(Q)^{n-1}\biggl(\fint_Q\mathcal{A}_2(t\nabla^m u)^p\biggr)^{2/p}
\\&\leq
C\sum_{j=-K}^\infty 2^{j(n-1-2n/p)}
\sum_{Q\in \mathcal{G}_j} 
\biggl(\int_Q\mathcal{A}_2(t\nabla^m u)^p\biggr)^{2/p}
.\end{align*} 
If $p\leq 2$ then 
\begin{equation*}\sum_{Q\in \mathcal{G}_j} 
\biggl(\int_Q\mathcal{A}_2(t\nabla^m u)^p\biggr)^{2/p}
\leq 
\biggl(\int_{\R^n}\mathcal{A}_2(t\nabla^m u)^p\biggr)^{2/p}
\end{equation*} 
and also $n-1-2n/p\leq -1$, so we may choose $K$ such that
\begin{equation*}\int_{\R^n}\int_{\sigma}^\infty\abs{\nabla^m u(x,t)}^2\,dt\,dx \leq \frac{C}{\sigma^{2n/p+1-n}}\doublebar{\mathcal{A}_2(t\nabla^m u)}_{L^p(\R^n)}^2.\end{equation*}
Thus, $u \in \dot W^2_m(\R^n\times(\sigma,\infty))$, albeit with norm that increases to infinity as $\sigma\to 0^+$. 
\end{rmk}

In a forthcoming paper, we hope to establish the bounds \eqref{B:eqn:S:lusin:p} and~\eqref{B:eqn:S:lusin:variant} for at least some values of $p'<2$; thus, we have formulated Theorems~\ref{B:thm:Dirichlet:1} and~\ref{B:thm:Dirichlet:2} so as to easily be able to improve the range of~$p$ in Theorems~\ref{B:thm:Dirichlet:1} and~\ref{B:thm:Dirichlet:2}.

The remainder of this section will be devoted to a proof of Theorems~\ref{B:thm:Dirichlet:1} and~\ref{B:thm:Dirichlet:2}.

Fix $\sigma>0$ and let $\mathcal{G}_\sigma$ be a grid of cubes of side-length~$\sigma/c$ for some large constant~$c$.
By Lemma~\ref{B:lem:slices}, if 
$p<p^+_1$ then
\begin{align*}
\int_{\R^n}\abs{\nabla^{m-1}\partial_\sigma v(x,\sigma)}^p\,dx 
&\leq 
	\sum_{Q\in\mathcal{G}_\sigma} \int_{Q}\abs{\nabla^{m-1} \partial_\sigma v(x,\sigma)}^p\,dx
\\&\leq 
	C\sigma^{-1}\sum_{Q\in\mathcal{G}_\sigma} \int_{2Q}\int_{\sigma-\sigma/4c}^{\sigma+\sigma/4c} \abs{\nabla^{m-1}\partial_\sigma  v(x,t)}^p\,dx\,dt
.\end{align*}
By H\"older's inequality or Theorem~\ref{B:thm:Meyers},
\begin{multline*}
\int_{\R^n}\abs{\nabla^{m-1}\partial_\sigma v(x,\sigma)}^p\,dx 
\\\leq 
	C\sigma^{n-p}\sum_{Q\in\mathcal{G}_\sigma} \biggl(\int_{4Q}\int_{\sigma-\sigma/2c}^{\sigma+\sigma/2c} \abs{\nabla^{m-1} \partial_\sigma v(x,t)}^2\frac{1}{\sigma^{n-1}}\,dx\,dt\biggr)^{p/2}
\end{multline*}
and by the definition~\eqref{B:eqn:lusin} of $\mathcal{A}_2$, if $c$ is large enough then
\begin{align*}
\int_{\R^n}\abs{\nabla^{m-1}\partial_\sigma v(x,\sigma)}^p\,dx 
&\leq 
	C\sigma^{n-p}\sum_{Q\in\mathcal{G}_\sigma} 
	\fint_Q \mathcal{A}_2(\1_{\sigma/2<t<3\sigma/2} t\nabla^{m} v(\,\cdot\,,t))^p
\\&= 
	C\sigma^{-p}
	\int_{\R^n} \mathcal{A}_2(\1_{\sigma/2<t<3\sigma/2} \sigma\nabla^{m} v)^p
.\end{align*}

Later in this paper we will use the fact that if $p<p_0^+$, then by the same argument,
\begin{align}
\label{B:eqn:Dirichlet:decay}
\int_{\R^n}\abs{\nabla^{m} v(x,\sigma)}^p\,dx 
&\leq 
	C\sigma^{-p}
	\int_{\R^n} \mathcal{A}_2(\1_{\sigma/2<t<3\sigma/2} \sigma\nabla^{m} v)^p
.\end{align}

So by the dominated convergence theorem, $\sigma\nabla^{m-1}\partial_\sigma v(\,\cdot\,,\sigma)\to 0$ as $\sigma\to \infty$ strongly in $L^p(\R^n)$. By the Caccioppoli inequality and Theorem~\ref{B:thm:Meyers}, if $k\geq 1$ is an integer then $\sigma^{k}\nabla^{m-1}\partial_\sigma^k v(\,\cdot\,,\sigma)\to 0$ (and in particular is bounded) in $L^p(\R^n)$ as $\sigma\to\infty$. Similarly, if $p<p^+_0$ and $k$ is large enough then $\sigma^{k}\nabla^m \partial_\sigma^{k} w(\,\cdot\,,\sigma)\to 0$ in $L^p(\R^n)$ as $\sigma\to \infty$.

Let $\arr g\in L^{p'}(\R^n)$ and $\arr h\in L^{p'}(\R^n)$ be bounded and compactly supported. 
Choose some $T>\tau >0$. We wish to bound the quantities
\begin{equation*}
\langle \arr g,\nabla^{m-1}v(\,\cdot\,,T)-\nabla^{m-1}v(\,\cdot\,,\tau )\rangle_{\R^n}
\quad\text{and}\quad
\langle \arr h,\nabla^{m}w(\,\cdot\,,T)-\nabla^{m}w(\,\cdot\,,\tau )\rangle_{\R^n}
\end{equation*}
in terms of $\tau $, $T$ and $\doublebar{\arr g}_{L^{p'}(\R^n)}$ or $\doublebar{\arr h}_{L^{p'}(\R^n)}$. Doing so will allow us to control $\nabla^{m-1}v(\,\cdot\,,T)-\nabla^{m-1}v(\,\cdot\,,\tau )$ or $\nabla^{m}w(\,\cdot\,,T)-\nabla^{m}w(\,\cdot\,,\tau )$; in particular, we will show that these quantities go to zero as $\tau \to\infty$ or $T\to 0^+$, and so we will see that $\nabla^{m-1}v$ or $\nabla^m w$ approaches a limit at $\infty$ and at zero.

Let $f(s)=\langle \arr g,\nabla^{m-1}v(\,\cdot\,,s)\rangle_{\R^n}$; observe that the $j$th derivative $f^{(j)}(s)$ of $f(s)$ satisfies $f^{(j)}(s)=\langle \arr g,\nabla^{m-1}\partial_s^{j} v(\,\cdot\,,s)\rangle_{\R^n}$. 
Let $\omega_0(s)=1$ if $\tau <s<T$ and $\omega_1(s)=0$ if $s>T$. 
Thus, 
\begin{align*}
\langle \arr g,\nabla^{m-1}v(\,\cdot\,,T)-\nabla^{m-1}v(\,\cdot\,,\tau )\rangle_{\R^n}
&=
	\int_\tau ^\infty
	\omega_0(s) f'(s)\,ds
.\end{align*}
Integrating from $\tau $ to $\infty$ will be somewhat simpler than integrating from $\tau $ to~$T$.
We wish to integrate by parts so that the right-hand side involves higher derivatives of $f(s)$.
Let $\omega_j(s)=\int_\tau^s \omega_{j-1}$. Using induction, it is straightforward to establish that
\begin{equation*}
\omega_j(s)\leq \begin{cases} 
	\frac{1}{j!} (s-\tau )^{j}, & \tau <s<T,\\ 
	\frac{1}{(j-1)!} (s-\tau )^{j-1}(T-\tau ), & T<s
.\end{cases}\end{equation*}
By our bound on~$\omega_{j}$ and by definition of $f(s)$, if $s$ is large enough then
\begin{equation*}\omega_{j}(s)\abs{f^{(j)}(s)}\leq C(j)\,s^{j}\,\doublebar{\arr g}_{L^{p'}(\R^n)} \doublebar{\nabla^{m-1} \partial_\dmn^j v(\,\cdot\,,s)}_{L^p(\R^n)}
\end{equation*}
and if $j\geq 1$, then by our above bounds on~$\doublebar{\nabla^{m-1} \partial_\dmn^k v(\,\cdot\,,s)}_{L^p(\R^n)}$, the right-hand side converges to zero as $s\to \infty$. Thus, we may integrate by parts and see that, for any $j\geq 0$,
\begin{align*}
\langle \arr g,\nabla^{m-1}v(\,\cdot\,,T)-\nabla^{m-1}v(\,\cdot\,,\tau )\rangle_{\R^n}
&=
	\int_\tau^\infty
	\omega_{2j}(s) f^{(2j+1)}(s)\,ds
\\&=
	\int_\tau^\infty
	\omega_{2j}(s) \langle \arr g,\nabla^{m-1}\partial_s^{2j+1} v(\,\cdot\,,s)\rangle_{\R^n}\,ds
.\end{align*}
Similarly,
\begin{align*}
\langle \arr h,\nabla^{m}w(\,\cdot\,,T)-\nabla^{m}w(\,\cdot\,,\tau )\rangle_{\R^n}
&=
	\int_\tau^\infty
	\omega_{2j}(s) \langle \arr h,\nabla^{m}\partial_s^{2j+1} w(\,\cdot\,,s)\rangle_{\R^n}\,ds
.\end{align*}

By formula~\eqref{B:eqn:S:S:vertical}, there is a constant rectangular matrix $\mathcal{O}^+$ such that
\begin{equation}
\label{B:eqn:S:S:variant}
\partial_t\s^{L} \arr g(x,t)=-\s^{L}_\nabla(\mathcal{O}^+\arr g)(x,t)
\end{equation}
for any array $\arr g$ of functions indexed by multiindices $\gamma$ with $\abs\gamma=m-1$. Then
\begin{align*}
\langle \arr g,\nabla^{m-1}v(\,\cdot\,,T)-\nabla^{m-1}v(\,\cdot\,,\tau )\rangle_{\R^n}
&=
	\int_\tau^\infty
	\omega_{2j}(s) \langle \mathcal{O}^+\arr g,\nabla^{m}\partial_s^{2j} v(\,\cdot\,,s)\rangle_{\R^n}\,ds
.\end{align*}

By Lemma~\ref{B:lem:control:lusin} with $\psi_s\equiv \mathcal{O}^+\arr g$ for all~$s$ and with $\omega(s)=\omega_{2j}(s)/s^{2j}$, and by formula~\eqref{B:eqn:S:S:variant}, we have that
\begin{multline*}
\abs{\langle \arr g,\nabla^{m-1}v(\,\cdot\,,T)-\nabla^{m-1}v(\,\cdot\,,\tau )\rangle_{\R^n}}
\\\begin{aligned}
&\leq 
	C\int_{\R^n}
	\mathcal{A}_2^-(\abs{t}^{j+1-2m}\partial_\dmn^{j-m+1}
	\s^{L^*} \arr g)(x)
	\,
	\mathcal{A}_2(t\,\Omega(t)\,\nabla^m v)(x)
	\,dx
\end{aligned}\end{multline*}
where $\Omega(s)$ satisfies the bounds
\begin{equation*}
\Omega(s)\leq C\begin{cases} 
	0, & s<\tau ,\\
	(1-\tau /s)^{2j}, & \tau <s<T,\\ 
	(1-\tau /s)^{2j-1}(T/s-\tau /s), & T<s
.\end{cases}\end{equation*}
By the Caccioppoli inequality and the bound \eqref{B:eqn:S:lusin:p}, if $j-m+1\geq k+m$, then $\mathcal{A}_2^-(\abs{t}^{j+1-2m}\partial_\dmn^{j-m+1}
	\s^{L^*} \arr g)\in L^{p'}(\R^n)$ with $L^{p'}$ norm at most $C\doublebar{\arr g}_{L^{p'}(\R^n)}$.

By assumption, $\mathcal{A}_2 (t\,{\nabla^m v})\in L^p(\R^n)$. Because $\Omega(s)$ is bounded, we have that
\begin{equation*}\mathcal{A}_2 \bigl(t\,\Omega(t)\,\nabla^m v\bigr)(x) \leq 
C\mathcal{A}_2 (t\,{\nabla^m v})(x).\end{equation*}
Furthermore, if $\mathcal{A}_2 (t\,{\nabla^m v})(x)<\infty$ (true for almost every $x\in\R^n$) then 
\begin{equation*}\mathcal{A}_2 \bigl(t\,\Omega(t)\,\nabla^m v\bigr)(x) \to 0 \quad\text{ as $\tau \to\infty$ or $T\to 0^+$.}\end{equation*}
By the dominated convergence theorem, this means that 
\begin{equation*}\doublebar[big]{\mathcal{A}_2 \bigl(t\,\Omega(t)\,\nabla^m v\bigr)}_{L^p(\R^n)} \to 0 \quad\text{ as $\tau \to\infty$ or $T\to 0^+$.}\end{equation*}
Thus, for any sequence of positive numbers $t_j$ that converge to either zero or infinity, the sequence $\{\nabla^{m-1} v(\,\cdot\,,t_j)\}_{j=1}^\infty$ is a Cauchy sequence in $L^p(\R^n)$, and so the limits $\arr p=\lim_{t\to \infty}\nabla^{m-1} v(\,\cdot\,,t)$ and $\arr f=\lim_{t\to \infty}\nabla^{m-1} v(\,\cdot\,,t)-\arr p$ exist. Furthermore, $\doublebar{\nabla^{m-1} v(\,\cdot\,,t)-\arr p}_{L^p(\R^n)}$ is bounded, uniformly in~$t$.

Similarly, the limits $\arr p'=\lim_{t\to \infty}\nabla^{m} w(\,\cdot\,,t)$ and $\arr f'=\lim_{t\to \infty}\nabla^{m} w(\,\cdot\,,t)-\arr p'$ exist. Furthermore, $\doublebar{\nabla^{m} w(\,\cdot\,,t)-\arr p'}_{L^p(\R^n)}$ is bounded, uniformly in~$t$.

It remains only to produce statements about the limits $\arr p$, $\arr p'$ at~$\infty$. 

If $p<p_0^+$, then by formula~\eqref{B:eqn:Dirichlet:decay}, $\nabla^m v(\,\cdot\,,t)\to 0$ in $L^p(\R^n)$ as $t\to \infty$, and so $\nabla_\pureH \arr p=0$ and so $\arr p$ is a constant array. But $\arr p$ is constant if and only if $\arr p=\nabla^{m-1} P$ for some polynomial $P$ of degree at most $m-1$, as desired.

If $p_0^+\leq p<p_1^+$, we will need a more complicated argument. By Lemma~\ref{B:lem:slices}, if $\Delta(x,c\tau)=\{y\in\R^n:\abs{x-y}<c\tau\}$ is a disk in $\R^n$ (and not in $\R^\dmn$) then
\begin{equation*}
\fint_{\Delta(x,\tau/2)} \abs{\nabla^m v(y,\tau)}^2\,dy
\leq
C \tau^{-1} \mathcal{A}_2(\1_{\tau/2<t<3\tau/2}t\nabla^m v)(x).
\end{equation*}
Notice that if $p=2$ then the left-hand side is simply $\doublebar{\nabla^m v(\,\cdot\,,\tau)}_{L^2(\R^n)}^2$.

If $\arr\varphi$ is an array of test functions, then
\begin{align*}
\int_{\R^n}\abs{\langle \arr\varphi,\nabla^m v(\,\cdot\,,\tau)\rangle}
&=
	C\int_{\R^n}\fint_{\Delta(x,\tau/2)} \abs{\langle \arr \varphi(y), \nabla^m v(y,\tau)\rangle}\,dy\, dx
\\&\leq 
	\frac{C}{\tau}\int_{\R^n}
	\biggl(\fint_{\Delta(x,\tau/2)} \abs{ \arr \varphi}^2\biggr)^{1/2}\mathcal{A}_2(\1_{\tau/2<t<3\tau/2}t\nabla^m v)(x) \, dx.
\end{align*}
Now, if $\varphi$ is smooth and compactly supported, then the function
\begin{equation*}f(x)=\biggl(\fint_{\Delta(x,\tau/2)} \abs{ \arr \varphi}^2\biggr)^{1/2}\end{equation*}
is also bounded and compactly supported, and in particular lies in $L^{p'}(\R^n)$. Thus, if $1\leq k\leq n$ and $\abs\gamma=m-1$, then $\abs{\langle \varphi,\partial_k \partial^\gamma v(\,\cdot\,,\tau)\rangle_{\R^n}}\to 0$ as $\tau\to \infty$. Therefore, $\abs{\langle \varphi,\partial_k p_\gamma\rangle_{\R^n}}= 0$; thus, $\arr p$ has a weak derivative that is everywhere zero, and so $\arr p$ must again be a constant.

We may strengthen this bound by using the Caccioppoli inequality: if $\ell\geq 0$ is an integer, then
\begin{multline}\label{B:eqn:local:2:global:p}
\tau^\ell\int_{\R^n}\abs{\langle\arr \varphi(x),\nabla^m\partial_\tau^\ell v(x,\tau)\rangle}\,dx
\\\leq 
	\frac{C}{\tau}\int_{\R^n}
	\biggl(\fint_{\Delta(x,\tau/2)} \abs{ \arr \varphi}^2\biggr)^{1/2}\mathcal{A}_2(\1_{\tau/2<t<3\tau/2}t\nabla^m v)(x) \, dx.
\end{multline}
We will need this bound in Section~\ref{B:sec:Neumann}; we observe that it is valid whenever $Lv=0$ in $\R^\dmn_+$, $1\leq p\leq \infty$ and the right-hand side is finite.

We now turn to $w$ and $\arr p'$.
By a similar argument, $\nabla^{m}\partial_t w(\,\cdot\,,t)\to 0$ and so $\nabla^{m-1}\partial_\perp w$ approaches a constant $\arr p'_1$. There is some polynomial $P_1$ with $\arr p'_1=\nabla^{m-1}\partial_\dmn P_1$ and $\nabla^m \partial_t P_1(x,t)=0$. We are left with $\arr p'_2=\lim_{t\to\infty}\nabla^m_\pureH w(\,\cdot\,,t)$. Since $w(\,\cdot\,,t)$ is locally in $\dot W^2_m(\R^n)$, we have that $\arr p'_2=\nabla_\pureH^m P_2$ for some function $P_2$ defined on $\R^n$. Thus, $\arr p = \nabla^m P$ where $P(x,t)=P_1(x,t)+P_2(x)$, and $\nabla^m \partial_t P(x,t)=0$, as desired.

We next check the claim $LP=0$.
Let $\varphi$ be smooth and compactly supported. Then
\begin{align*}
\langle \nabla^m\varphi,\mat A\nabla^m P\rangle_{\R^\dmn}
&= 
\int_{-\infty}^\infty \langle \nabla^m \varphi(\,\cdot\,,t),\mat A\nabla^m P\rangle_{\R^n}\,dt 
\\&= 
\lim_{s\to \infty}\int_{-\infty}^\infty \langle \nabla^m \varphi(\,\cdot\,,t),\mat A\nabla^m w(\,\cdot\,,s+t)\rangle_{\R^n}\,dt 
\\&= 
\lim_{s\to \infty}\langle \nabla^m \varphi_{-s},\mat A\nabla^m w\rangle_{\R^\dmn_+}
=0
\end{align*}
because $Lw=0$. (Here $\varphi_{-s}(x,t)=\varphi(x,t-s)$; if $\varphi$ is supported in $\R^n\times(-T,T)$ then $\varphi_{-s}$ is supported in $\R^n\times(s-T,s+T)$.) Thus, $LP=0$ as well.

Finally, suppose that $\nabla^m w(\,\cdot\,,t)\in L^p(\R^n)$ for some $t>0$. Then $\nabla^m P(\,\cdot\,,t)\in L^p(\R^n)$ as well. Because $\nabla^{m-1} \partial_\dmn P=\nabla^{m-1}\partial_\dmn P_1$ is constant, it must equal zero if $p<\infty$.
Thus $\nabla^m P=\nabla^m_\pureH P$ is constant in~$t$. 

Let $\mat A_\pureH$ satisfy
$(A_\pureH)_{\alpha\beta}=A_{\alpha\beta}$ if $\alpha_\dmn=\beta_\dmn=0$.
Let $L_\pureH$ be the elliptic operator associated to~$\mat A_\pureH$; the $L_\pureH$ acts on functions defined on $\R^n$ rather than $\R^\dmn$.

Then $L_\pureH P=0$ in $\R^n$ (regarding $P$ as a function of $\R^n$). If $p<p^+_{L_\pureH}$, then by Theorem~\ref{B:thm:Meyers} applied in $\R^n$, there is some $q$ with $p<q<p^+_{L_\pureH}$ and such that
\begin{equation*}\doublebar{\nabla_\pureH^m P}_{L^q(\Delta(0,R))} \leq C R^{n/q-n/p}\doublebar{\nabla_\pureH^m P}_{L^p(\Delta(0,2R))}\end{equation*}
for all $R>0$, where $\Delta(x,r)=\{y\in\R^n:\abs{x-y}<r\}$. Letting $R\to \infty$ we see that $\nabla_\pureH^m P=0$ almost everywhere, as desired.


\section{The Neumann boundary values of a solution}
\label{B:sec:Neumann}

In this section we will prove results pertaining to the Neumann boundary values as defined by formula~\eqref{B:eqn:Neumann:E}, that is, defined in terms of a specific extension operator~$\mathcal{E}$. Specifically, we will prove the following two theorems.

\begin{thm}\label{B:thm:Neumann:1}
Suppose that $L$ is an elliptic operator associated with coefficients $\mat A$ that are $t$-independent in the sense of formula~\eqref{B:eqn:t-independent} and satisfy the ellipticity conditions \eqref{B:eqn:elliptic} and~\eqref{B:eqn:elliptic:bounded}.
Let $v\in \dot W^2_{m,loc}(\R^\dmn_+)$ and suppose that $Lv=0$ in $\R^\dmn_+$.

Suppose that ${\mathcal{A}_2(t\nabla^m v)}\in {L^p(\R^n)}$ for some $1<p<\infty$. 
Further assume that for any $\sigma>0$ we have that $\nabla^m v\in L^2(\R^n\times(\sigma,\infty))$.

Then for all $\varphi$ smooth and compactly supported, we have that
\begin{equation*}{\langle \mat A\nabla^m v(\,\cdot\,,t),\nabla^m \mathcal{E}\varphi(\,\cdot\,,t)\rangle_{\R^n}}\end{equation*}
represents an absolutely convergent integral for any fixed $t>0$ and is continuous in~$t$.

Furthermore,
\begin{multline*}
\sup_{0<\varepsilon<T}
\abs[bigg]{
\int_\varepsilon^T
\langle \mat A\nabla^m v(\,\cdot\,,t),\nabla^m \mathcal{E}\varphi(\,\cdot\,,t)\rangle_{\R^n}\,dt 
}
\\\leq 
C \doublebar{\nabla_\pureH \Tr_{m-1}^+\varphi}_{L^{p'}(\R^n)}
\doublebar{\mathcal{A}_2(t\nabla^m v)}_{L^p(\R^n)}
\end{multline*}
and the limit
\begin{equation*}\lim_{\varepsilon\to 0^+} \lim_{T\to \infty} \int_\varepsilon^T
\langle \mat A\nabla^m v(\,\cdot\,,t),\nabla^m \mathcal{E}\varphi(\,\cdot\,,t)\rangle_{\R^n}\,dt 
\end{equation*}
exists,
and so we have the bound
\begin{align*}
\abs{\langle\M_{\mat A}^+v,\Tr_{m-1}\varphi\rangle_{\R^\dmn_+}}
&\leq C \doublebar{\nabla_\pureH \Tr_{m-1}^+\varphi}_{L^{p'}(\R^n)}
\doublebar{\mathcal{A}_2(t\nabla^m v)}_{L^p(\R^n)}
.\end{align*}
\end{thm}

\begin{thm}\label{B:thm:Neumann:2}
Suppose that $L$ is an elliptic operator associated with coefficients $\mat A$ that are $t$-independent in the sense of formula~\eqref{B:eqn:t-independent} and satisfy the ellipticity conditions \eqref{B:eqn:elliptic} and~\eqref{B:eqn:elliptic:bounded}.
Let $w\in \dot W^2_{m,loc}(\R^\dmn_+)$ and suppose that $Lw=0$ in $\R^\dmn_+$.

Suppose that ${\mathcal{A}_2(t\nabla^m\partial_t w)}\in {L^p(\R^n)}$  for some $1<p<\infty$.
Further assume that for any $\sigma>0$ we have that $\nabla^m \partial_\dmn w\in L^2(\R^n\times(\sigma,\infty))$. 
Finally, assume that
\begin{equation*}\sup_{\tau>0}\biggl(\int_{\R^n}  \biggl(\fint_{B((x,\tau),\tau/2)} \abs{\nabla^m w}^2\biggr)^{p/2}\biggr)^{1/p} = C_0<\infty.\end{equation*}

Then for all $\varphi$ smooth and compactly supported in $\R^\dmn$ we have that the bound
\begin{equation*}
\abs{\langle \M_{\mat A}^+ w,\Tr_{m-1}\varphi\rangle_{\R^n}}
\leq C \doublebar{\Tr_{m-1}^+\varphi}_{L^{p'}(\R^n)}(C_0+
\doublebar{\mathcal{A}_2(t\nabla^m \partial_t w)}_{L^p(\R^n)})
\end{equation*}
is valid. Furthermore, we have that
\begin{equation}
\label{B:eqn:w:absolute}
\int_0^\infty \int_{\R^n}\abs{\langle \mat A(x)\nabla^m w(x,t),\nabla^m\mathcal{E}\varphi(x,t)\rangle}\,dx\,dt<\infty \end{equation}
and 
\begin{equation}
\label{B:eqn:Neumann:regular}
\langle \M_{\mat A}^+ w,\Tr_{m-1}\varphi\rangle_{\R^n} = \langle \mat A\nabla^m w,\nabla^m\varphi\rangle_{\R^\dmn_+}
\end{equation}
for any smooth, compactly supported extension of $\Tr_{m-1}\varphi$; that is, the Neumann boundary values may be defined in terms of arbitrary $C^\infty_0$ extensions and not the distinguished extension~$\mathcal{E}\varphi$. 
\end{thm}

Before proving these theorems, we make two remarks; these remarks may assist in applying Theorems~\ref{B:thm:Neumann:1} and~\ref{B:thm:Neumann:2}.

\begin{rmk} \label{B:rmk:uniform}
We comment on the appearance in Theorem~\ref{B:thm:Neumann:2} of the term 
\begin{equation*}\sup_{\tau>0}\biggl(\int_{\R^n} \biggl(\fint_{B((x,\tau),\tau/2)} \abs{\nabla^m w}^2\biggr)^{p/2}\biggr)^{1/p}.\end{equation*}

If $p<p^+_L$, where $p^+_L$ is as in Theorem~\ref{B:thm:Meyers}, then 
\begin{equation*}\sup_{\tau>0}\int_{\R^n} \biggl(\fint_{B(x,\tau),\tau/2} \abs{\nabla^m w}^2\biggr)^{p/2}\,dx\leq C\sup_{\tau>0}\doublebar{\nabla^m w(\,\cdot\,,\tau)}_{L^p(\R^n)}^p
\end{equation*}
and so if $p'$ is such that the condition~\eqref{B:eqn:S:lusin:variant} is valid, then by Theorem~\ref{B:thm:Dirichlet:2} we have that
\begin{equation*}\sup_{\tau>0}\int_{\R^n} \biggl(\fint_{B((x,\tau),\tau/2)} \abs{\nabla^m w}^2\biggr)^{p/2}
\leq C\doublebar{\mathcal{A}_2(t\nabla^m\partial_t w)}_{L^p(\R^n)}^p.\end{equation*}
provided $\doublebar{\nabla^m w(\,\cdot\,,\tau)}_{L^p(\R^n)}<\infty$ for at least one value of $\tau>0$.

As mentioned in the introduction, this term appears in other ways in the theory; for example, if $\widetilde N$ is the modified nontangential maximal function introduced in \cite{KenP93}, then
\begin{equation*}\sup_{\tau>0}\int_{\R^n} \biggl(\fint_{B((x,\tau),\tau/2)} \abs{\nabla^m w}^2\biggr)^{p/2}
\leq C\doublebar{\widetilde N(\nabla^m w)}_{L^p(\R^n)}^p.
\end{equation*}

\end{rmk}

\begin{rmk}\label{B:rmk:Neumann:p<2}
As in Section~\ref{B:sec:Dirichlet}, if $p\leq 2$, then finiteness of  $\doublebar{\mathcal{A}_2(t\nabla^m v)}_{L^p(\R^n)} $ or $\doublebar{\mathcal{A}_2(t\nabla^m \partial_t w)}_{L^p(\R^n)} $ implies that $\nabla^m v\in L^2(\R^n\times(\sigma,\infty))$ or $\nabla^m \partial_\dmn w\in L^2(\R^n\times(\sigma,\infty))$, respectively, for any $\sigma>0$. 

Thus, if $1<p\leq 2$, then $v$ satisfies the conditions of Theorem~\ref{B:thm:Neumann:1} provided only that $\mathcal{A}_2(t\nabla^m v)\in L^p(\R^n)$ and $Lv=0$ in $\R^\dmn_+$.

Similarly, by Remark~\ref{B:rmk:uniform}, if $1<p\leq 2$ then $w-P$ satisfies the conditions of Theorem~\ref{B:thm:Neumann:2} provided $\mathcal{A}_2(t\nabla^m \partial_t w)\in L^p(\R^n)$ and $Lw=0$ in $\R^\dmn_+$, where $P$ is as in Theorem~\ref{B:thm:Dirichlet:2}.
\end{rmk}

We will devote the remainder of this section to a proof of these two theorems. 

We begin with the following estimates on $\mathcal{Q}_t^m$.

\begin{lem}\label{B:lem:semigroup:area}
Let $0\leq j\leq m$ and let $\ell\geq j$ be an integer. Let $\gamma$ be a multiindex with $\gamma_\dmn=0$ and $\abs\gamma\leq \ell$. 

If $1\leq r\leq {p'}\leq \infty$, then
\begin{equation}\label{B:eqn:Q:Lp}
\doublebar{t^{\ell-j} \partial_\pureH^\gamma \partial_t^{\ell-\abs\gamma}\mathcal{Q}_t^m\psi}_{L^{p'}(\R^n)} 
\leq 
C_{p',r} t^{n/{p'}-n/r}\doublebar{\nabla_\pureH^j\psi}_{L^r(\R^n)}
\end{equation}
for any $t>0$ and $\psi\in \dot W^r_j(\R^n)$.

If $1<p'<\infty$, and if
$\ell>\abs\gamma$ or $\ell=\abs\gamma>j$, then 
then
\begin{equation}
\label{B:eqn:Q:area}
\doublebar{\mathcal{A}_2(t^{\ell-j} \partial_\pureH^\gamma \partial_t^{\ell-\abs\gamma}\mathcal{Q}_t^m\psi)}_{L^{p'}(\R^n)} \leq C \doublebar{\nabla_\pureH^j\psi}_{L^{p'}(\R^n)}\end{equation}
for any $\psi\in \dot W^{p'}_j(\R^n)$.

\end{lem}

\begin{proof}
For any Schwartz function $\eta$, let $\eta_t(y)=t^{-n} \eta(y/t)$.
Recall that $\mathcal{Q}_t^m=e^{-(-t^2\Delta_\pureH)^m}$; a straightforward argument using the Fourier transform establishes that $\mathcal{Q}_t^m f(x)=\theta_t*f(x)$ for some Schwartz function~$\theta$. 

Observe that $\partial_t\mathcal{Q}_t^m= -2m\,t^{2m-1}(-\Delta_\pureH)^{m}\mathcal{Q}_t^m$. Thus, there are some constants $C_{\ell,m,\gamma,\zeta}$ such that
\begin{align*}
t^{\ell-j} \partial_\pureH^\gamma \partial_t^{\ell-\abs\gamma}\mathcal{Q}_t^m\psi(x)
&= 
\sum_{\mathllap{2m}\leq \abs\zeta\leq \mathrlap{2m(\ell-\abs\gamma)}} t^{\abs\zeta+\abs\gamma-j} 
C_{\ell,m,\gamma,\zeta}\partial_\pureH^{\zeta+\gamma}\mathcal{Q}_t^m\psi(x)
&&\text{if $\ell>\abs\gamma$}
,\\
t^{\ell-j} \partial_\pureH^\gamma \partial_t^{\ell-\abs\gamma}\mathcal{Q}_t^m\psi(x)
&= 
t^{\abs\gamma-j} \partial_\pureH^\gamma \mathcal{Q}_t^m\psi(x)
&&\text{if $\ell=\abs\gamma$}
.\end{align*}
Notice that the purely horizontal derivatives may be chosen to fall on either $\psi$ or the convolution kernel of~$\mathcal{Q}_t^m$, and, furthermore, if either $\ell>\abs\gamma$ or $\ell=\abs\gamma\geq j$ then there are at least $j$ such derivatives. Thus, we have that
\begin{equation}\label{B:eqn:Q:gradient}
t^{\ell-j}\partial_\pureH^\gamma \partial_t^{\ell-\abs\gamma} \mathcal{Q}^{m}_t \psi(x)
=\sum_{\abs\delta=j}\eta^\delta_t*\partial^\delta\psi(x)
=\arr \eta_t*\nabla_\pureH^j\psi(x)
\end{equation}
for some array of Schwartz functions $\arr \eta$ depending on $\gamma$, $\ell$, $m$ and~$j$. 

Observe that $\doublebar{\eta_t}_{L^s(\R^n)} = C_s t^{n/s-n}$ for some constant~$C_s$. It is well known that, if $1\leq r\leq {p'}\leq\infty$, then 
\begin{equation*}\doublebar{\arr\eta_t*\nabla_\pureH^j\psi}_{L^{p'}(\R^n)}\leq \doublebar{\arr\eta_t}_{L^s(\R^n)}\doublebar{\nabla_\pureH^j\psi}_{L^r(\R^n)},\end{equation*}
where $1/{p'}+1=1/s+1/r$.
Combining these estimates yields the bound \eqref{B:eqn:Q:Lp}.

Let $\rho$ be a Schwartz function that satisfies $\int \rho=0$.
It is a straightforward consequence of vector-valued $Tb$ theory (see, for example, \cite[Chapter~I, Section~6.4]{Ste93}) that \begin{equation*}\doublebar{\mathcal{A}_2(\rho_t*f) }_{L^{p'}(\R^n)} \leq C(p') \doublebar{f}_{L^{p'}(\R^n)}\end{equation*}
for any $1<p'<\infty$.
Thus, to establish the bound~\eqref{B:eqn:Q:area}, it suffices to show that $\arr \eta$ integrates to zero.
To show that $\arr \eta$ integrates to zero, it suffices to show that, if $p_\delta(x)=x^\delta$ for some $\abs\delta=j$, so that $\nabla^j_\pureH p_\delta =\delta!\,\arr e_\delta$, then 
\begin{equation*}t^{\ell-j}\partial_\pureH^\gamma \partial_t^{\ell-\abs\gamma} \mathcal{Q}^{m}_t p_\delta(x)=0.\end{equation*}

But 
\begin{align*}
\mathcal{Q}^{m}_t p_\delta(x)=\theta_t*p_\delta(x)
&=\int (x-ty)^\delta \theta(y)\,dy = 
\sum_{\zeta\leq \delta} \frac{\delta!}{\zeta!(\delta-\zeta)!} \, x^\zeta \, t^{\abs{\delta-\zeta}} \int y^{\delta-\zeta}\,\theta(y)\,dy
\\&=
\sum_{\zeta\leq \delta} C_{\zeta,\delta}\, x^\zeta \,t^{j-\abs{\zeta}} \end{align*}
where the sum is over multiindices $\zeta$ with $\zeta_j\leq \delta_j$ for all $1\leq j\leq \dmn$.

Let $C_{\zeta,\delta}=0$ if $\abs\zeta\leq j$ but $\zeta\not\leq \delta$.

Recall that if $1\leq k\leq 2m-1$, then $\partial_t^k \mathcal{Q}^{m}_t \big\vert_{t=0}=0$. Thus,
\begin{align*}
0=\partial_t^k\mathcal{Q}^{m}_t p_\delta(x)\big\vert_{t=0}
&=
{k!}\sum_{\abs\zeta = j-k} C_{\zeta,\delta}\, x^\zeta 
\end{align*}
for any 
$1\leq k\leq j$.

Now, for more general~$t$, we compute
\begin{align*}
\partial_t^{\ell-\abs\gamma}\mathcal{Q}^{m}_t p_\delta(x)
&=
\sum_{k=\ell-\abs\gamma}^j\partial_t^{\ell-\abs\gamma}t^{k} 
\sum_{\abs\zeta=j-k} C_{\zeta,\delta}\, x^\zeta 
=
\partial_t^{\ell-\abs\gamma}t^{0} \, 
 C_{\delta,\delta}\, x^\delta 
.\end{align*}
This is zero whenever $\ell>\abs\gamma$. If $\ell=\abs\gamma$, then
\begin{align*}
\partial_\pureH^\gamma\partial_t^{\ell-\abs\gamma}\mathcal{Q}^{m}_t p_\delta(x)
&=
 C_{\delta,\delta}\, \partial_\pureH^\gamma x^\delta 
\end{align*}
which is zero if $\abs\gamma>\abs\delta=j$.
\end{proof}

Next, we prove the following lemma; this is the usage of the bound~\eqref{B:eqn:local:2:global:p} most applicable to the present case.
\begin{lem}\label{B:lem:Neumann:slices}
Let $L$ be as in Theorem~\ref{B:thm:Neumann:1}. Suppose that $Lu=0$ in $\R^\dmn_+$. If $\psi$ is smooth and compactly supported, and if $0\leq j\leq m$, $\ell\geq j$ and $k\geq 0$ are integers, then
\begin{multline*}
\int_{\R^n}
\abs{
	\tau^{\ell-j+k+1}\nabla^\ell \mathcal{Q}^{m}_\tau  \psi(x)
	\,\overline{A_{\gamma\beta}(x)\, \nabla^m\partial_\tau^k u(x,\tau)}\,}\,dx
\\\leq
	C \doublebar{\nabla_\pureH^j \psi}_{L^{p'}(\R^n)}	
	\doublebar{\mathcal{A}_2( \1_{\tau/2<t<3\tau/2} t\nabla^m u)}_{L^p(\R^n)}
.\end{multline*}
Furthermore, if $1\leq r \leq p'$, then
\begin{multline*}
\int_{\R^n}
\abs{
	\tau^{\ell-j+k+1}\nabla^{\ell-j}\nabla_{\pureH}^j \mathcal{Q}^{m}_\tau  \psi(x)
	\,\overline{A_{\gamma\beta}(x)\, \nabla^m\partial_\tau^k u(x,\tau)}\,}\,dx
\\\leq
	C \tau^{n/p'-n/r}\doublebar{\nabla_\pureH^j\psi}_{L^r(\R^n)}	
	\doublebar{\mathcal{A}_2( \1_{\tau/2<t<3\tau/2} t\nabla^m u)}_{L^p(\R^n)}
.\end{multline*}
\end{lem}

This lemma has obvious applications if $u=v$ or $u=\partial_t w$.
We remark that it may also be applied with $u=w$, because
\begin{equation*}\sup_{\tau>0}\frac{1}{\tau}\doublebar{\mathcal{A}_2(t\1_{\tau/2<t<3\tau/2} \nabla^m w)}_{L^p(\R^n)}^p
\leq C\sup_{\tau>0}\int_{\R^n} \biggl(\fint_{B((x,\tau),\tau/2)} \abs{\nabla^m w}^2\biggr)^{p/2}.\end{equation*}

\begin{proof}[Proof of Lemma~\ref{B:lem:Neumann:slices}]
By formula~\eqref{B:eqn:Q:gradient},
\begin{equation*}\sup_{\abs{x-y}<\tau/2} \abs{\tau^{\ell-j}\nabla^\ell \mathcal{Q}_\tau^m\psi(y)} \leq C\mathcal{M}(\nabla_\pureH^j\psi)(x)\end{equation*}
for any $z\in Q$, where $\mathcal{M}$ denotes the Hardy-Littlewood maximal function. Recall that $\mathcal{M}$ is bounded on $L^{p'}(\R^n)$ for any $1<p<\infty$. 
The first result then follows from the bound~\eqref{B:eqn:local:2:global:p}.

Because $\mathcal{Q}_t^m$ is a semigroup and commutes with horizontal derivatives, we have that if $z\in Q$ then
\begin{equation*}\sup_{x\in Q} \abs{\tau^{\ell-j}\nabla^{\ell-j}\nabla_\pureH^j \mathcal{Q}_\tau^m\psi_j(x)} \leq 
C\mathcal{M}(\nabla_\pureH^j\mathcal{Q}_{\tau/2}^m\psi_j)(z)
\end{equation*}
and again by boundedness of $\mathcal{M}$ and the bound~\eqref{B:eqn:local:2:global:p} we have that
\begin{multline*}
\int_{\R^n}
\abs{
	\tau^{\ell-j+k+1}\nabla^{\ell-j}\nabla_\pureH^j  \mathcal{Q}^{m}_\tau  \psi(x)
	\,\overline{A_{\gamma\beta}(x)\, \nabla^m\partial_\tau^k u(x,\tau)}\,}\,dx
\\\leq
	C \doublebar{\nabla_\pureH^j\mathcal{Q}^m_{\tau/2}\psi}_{L^{p'}(\R^n)}	
	\doublebar{\mathcal{A}_2(t \1_{\tau/2<t<3\tau/2} \nabla^m u)}_{L^p(\R^n)}
.\end{multline*}

Now, by the bound~\eqref{B:eqn:Q:Lp}, we have that if $1<r<p'$ then $\doublebar{\nabla_\pureH^j\mathcal{Q}_{\tau/2}^m\psi}_{L^{p'}(\R^n)} 
\leq 
C_{p',r} \tau^{n/{p'}-n/r}\doublebar{\nabla_\pureH^j\psi}_{L^r(\R^n)}
$.
This completes the proof of the second estimate.
\end{proof}

We now prove Theorems~\ref{B:thm:Neumann:1} and~\ref{B:thm:Neumann:2}. We begin with the terms that require different arguments in the two cases; we will conclude this section by bounding a term that arises in both cases.

\begin{lem}
\label{B:lem:Neumann:Dirichlet}
Let $v$ be as in Theorem~\ref{B:thm:Neumann:1}.
Then 
\begin{equation*}{\langle \mat A\nabla^m v(\,\cdot\,,t),\nabla^m \mathcal{E}\varphi(\,\cdot\,,t)\rangle_{\R^n}}\end{equation*}
represents an absolutely convergent integral and is continuous in~$t$.

Furthermore, let $\psi_j(x)=\varphi_{m-j}(x)=\partial_\dmn^{m-j}\varphi(x,0)$, so \begin{equation*}\frac{1}{C}\doublebar{\nabla_\pureH \Tr_{m-1}^+\varphi}_{L^{p'}(\R^n)}\leq
\sum_{j=1}^m\doublebar{\nabla_\pureH^j\psi_j}_{L^{p'}(\R^n)}\leq C \doublebar{\nabla_\pureH \Tr_{m-1}^+\varphi}_{L^{p'}(\R^n)}.\end{equation*}
Suppose that $\doublebar{\mathcal{A}_2(t\,\nabla^m v)}_{L^p(\R^n)}<\infty$ for some $1<p<\infty$.
Then 
\begin{multline*}\langle \mat A\nabla^m v(\,\cdot\,,t), \nabla^m \mathcal{E}\varphi(\,\cdot\,,t)\rangle_{\R^n} 
\\=
	OK(t,\varphi,v)+
\sum_{j=1}^{m}\sum_{\abs\beta=m}
	\sum_{\substack{\abs\gamma=j\\\gamma_\dmn=0}}
	\frac{1}{(m-j)!}
	\int_{\R^n}
	\partial_\pureH^\gamma \mathcal{Q}^{m}_t \psi_{j}(x)\,\overline{A_{\gamma\beta}(x)\,\partial^\beta v(x,t)}\,dx
\end{multline*}
where $A_{\gamma\beta}=A_{\tilde \gamma \beta}$ for $\tilde \gamma=\gamma+(m-\abs\gamma)\vec e_\dmn$,
for some term $OK(t,\varphi,v)$ that satisfies the bound
\begin{equation*}
\int_0^\infty \abs{OK(t,\varphi,v)}\,dt\leq
C \doublebar{\nabla_\pureH \Tr_{m-1}^+\varphi}_{L^{p'}(\R^n)} \doublebar{\mathcal{A}_2(t\,\nabla^m v)}_{L^p(\R^n)}.\end{equation*}
\end{lem}

\begin{proof}
Observe that by the definition \eqref{B:eqn:Neumann:extension} of~$\mathcal{E}$,
\begin{multline*}
\langle \mat A\nabla^m v(\,\cdot\,,t), \nabla^m \mathcal{E}\varphi(\,\cdot\,,t)\rangle_{\R^n} 
= 
	\sum_{j=1}^{m}\sum_{\abs\beta=m}
	\sum_{\substack{\abs\gamma\leq m\\\gamma_\dmn=0}}
	\frac{1}{(m-j)!}
	\\\times
	\int_{\R^n}\partial_\pureH^\gamma \partial_t^{m-\abs\gamma} (t^{m-j}  \mathcal{Q}^{m}_t  \psi_j(x))\,\overline{A_{\gamma\beta}(x)\,\partial^\beta v(x,t)}\,dx
.\end{multline*}
By the product rule,
\begin{multline*}
\langle \mat A\nabla^m v(\,\cdot\,,t), \nabla^m \mathcal{E}\varphi(\,\cdot\,,t)\rangle_{\R^n} 
= 
	\sum_{j=1}^{m}\sum_{\abs\beta=m}
	\sum_{\substack{\abs\gamma\leq m\\\gamma_\dmn=0}}
	\sum_{\ell=\max(j,\abs\gamma)}^m
	\frac{(m-\abs\gamma)!}{(\ell-\abs\gamma)!(m-\ell)!^2} 
	\\\times
	\int_{\R^n}
	t^{\ell-j} \partial_t^{\ell-\abs\gamma}\partial_\pureH^\gamma \mathcal{Q}^{m}_t  \psi_j(x)\,\overline{A_{\gamma\beta}(x)\,\partial^\beta v(x,t)}\,dx
.\end{multline*}
By Lemma~\ref{B:lem:Neumann:slices}, the integral is absolutely convergent and has absolute value at most 
\begin{equation*}Ct^{-1} \doublebar{\nabla_\pureH^j \psi_j}_{L^{p'}(\R^n)}	
	\doublebar{\mathcal{A}_2(t  \nabla^m u)}_{L^p(\R^n)}.\end{equation*}
Furthermore,
\begin{multline*}\int_{\R^n}\abs[bigg]{\frac{d}{dt}\biggl(
	t^{\ell-j} \partial_t^{\ell-\abs\gamma}\partial_\pureH^\gamma \mathcal{Q}^{m}_t  \psi_j(x)\,\overline{A_{\gamma\beta}(x)\,\partial^\beta v(x,t)}\biggr)}\,dx
\\\leq 
Ct^{-2} \doublebar{\nabla_\pureH^j \psi_j}_{L^{p'}(\R^n)}	
	\doublebar{\mathcal{A}_2(t  \nabla^m u)}_{L^p(\R^n)}
\end{multline*}
and so the integral over $\R^n$ is continuous (and in fact differentiable) in~$t$.

By formula~\eqref{B:eqn:area:0}, H\"older's inequality and the definition~\eqref{B:eqn:lusin} of $\mathcal{A}_2^a$, if $a>0$ and if $F$ and $G$ are nonnegative functions then
\begin{equation}
\label{B:eqn:area:2}
\int_{\R^n}\int_0^\infty F(x,t)\,G(x,t)\,dt\,dx 
= 
\frac{C_n}{a^n}
\int_{\R^n}
\mathcal{A}_2^a(F)(x)
\,
\mathcal{A}_2^a(t \,G)(x)
\,dx.\end{equation}
Thus,
\begin{multline*}
\int_{\R^n}\int_0^\infty 
	\abs[big]{t^{\ell-j}\,
	\partial_\pureH^\gamma \partial_t^{\ell-\abs\gamma} \mathcal{Q}^{m}_t \psi_{j}(x)\,\overline{A_{\gamma\beta}(x)\,\partial^\beta v(x,t)}}\,dt\,dx
\\\begin{aligned}
&\leq
	C
	\int_{\R^n} 
	\mathcal{A}_2(t^{\ell-j} \partial_\pureH^\gamma \partial_t^{\ell-\abs\gamma} \mathcal{Q}^{m}_t \psi_{j})(y)
	\,
	\mathcal{A}_2(t\,\partial^\beta v)(y)
	\,dy
.\end{aligned}\end{multline*}
By the bound~\eqref{B:eqn:Q:area}, if $\ell>\abs\gamma$ or $\ell=\abs\gamma>j$, then \begin{equation*}\doublebar{\mathcal{A}_2(t^{\ell-j} \partial_\pureH^\gamma \partial_t^{\ell-\abs\gamma} \mathcal{Q}^{m}_t \psi_{j})}_{L^{p'}(\R^n)} \leq \doublebar{\nabla_\pureH^j \psi_j}_{L^{p'}(\R^n)}.\end{equation*}
Thus, we need only consider the $\abs\gamma=j=\ell$ term; in other words,
\begin{multline*}
\langle \mat A\nabla^m v(\,\cdot\,,t), \nabla^m \mathcal{E}\varphi(\,\cdot\,,t)\rangle_{\R^n} 
\\= 
	OK(t)+\sum_{j=1}^{m}\sum_{\abs\beta=m}
	\sum_{\substack{\abs\gamma=j\\\gamma_\dmn=0}}
	\frac{1}{(m-j)!}
	\int_{\R^n}
	\partial_\pureH^\gamma \mathcal{Q}_t^m \psi_j(x) \,\overline{A_{\gamma\beta}(x)\,\partial^\beta v(x,t)}\,dx
.\end{multline*}
where the term $OK$ satisfies
\begin{equation*}
\int_0^\infty \abs{OK(t)}\,dt\leq C \doublebar{\nabla_\pureH \Tr_{m-1}^+\varphi}_{L^{p'}(\R^n)}\doublebar{\mathcal{A}_2(t\,\nabla^m v)}_{L^p(\R^n)}.\end{equation*}
This completes the proof.
\end{proof}

\begin{lem}
\label{B:lem:Neumann:regularity}
The bounds \eqref{B:eqn:w:absolute} and \eqref{B:eqn:Neumann:regular} are valid.

Furthermore, let $\psi_j(x)=\varphi_{m-j-1}(x)=\partial_\dmn^{m-j-1}\varphi(x,0)$, so \begin{equation*}\frac{1}{C}\doublebar{\Tr_{m-1}^+\varphi}_{L^{p'}(\R^n)} \leq\sum_{j=0}^{m-1}\doublebar{\nabla_\pureH^j\psi_j}_{L^{p'}(\R^n)}\leq C \doublebar{\Tr_{m-1}^+\varphi}_{L^{p'}(\R^n)}.\end{equation*}
Then for any $0<\varepsilon<T$ we have that
\begin{multline*}\int_\varepsilon^T \langle \mat A\nabla^m w(\,\cdot\,,t), \nabla^m \mathcal{E}\varphi(\,\cdot\,,t)\rangle_{\R^n}\,dt
-
OK_{\varepsilon,T}(w,\varphi)
\\ =	-\sum_{j=0}^{m-1}\sum_{\abs\beta=m}
	\sum_{\substack{\abs\gamma=j\\\gamma_\dmn=0}}
	\int_\varepsilon^T \int_{\R^n}
	\partial_\pureH^\gamma \mathcal{Q}^{m}_t  \psi_j(x)
	\,\overline{A_{\gamma\beta}(x)\,\partial^\beta \partial_t w(x,t)}\,dx\,dt
\end{multline*}
for some term $OK_{\varepsilon,T}(w,\varphi)$ that satisfies the bound
\begin{align*}
\abs{OK_{\varepsilon,T}(w,\varphi)}
&\leq
C \doublebar{\Tr_{m-1}^+\varphi}_{L^{p'}(\R^n)} \doublebar{\mathcal{A}_2(t\,\nabla^m \partial_t w)}_{L^p(\R^n)}
\\&\quad+C \doublebar{\Tr_{m-1}^+\varphi}_{L^{p'}(\R^n)} \sup_{\tau>0} \doublebar{\mathcal{A}_2(\1_{\tau/2<t<3\tau/2} \nabla^m w)}_{L^p(\R^n)}
.\end{align*}
\end{lem}

\begin{proof}
We begin with the bound \eqref{B:eqn:w:absolute}.
Observe that
\begin{align*}\nabla^m\mathcal{E}\varphi(x,t) &=\sum_{j=0}^{m-1}\nabla^m\biggl(\frac{t^{m-j-1}\mathcal{Q}^m_t \psi_j(x)}{(m-j-1)!} \biggr)
=
\sum_{j=0}^{m-1}
\sum_{\ell=j+1}^{m} 
C_{m,j,\ell} t^{\ell-j-1}\nabla^\ell\mathcal{Q}^m_t \psi_j(x)
.\end{align*}
By Lemma~\ref{B:lem:Neumann:slices} and the following remarks, if $1\leq r\leq p'$ then
\begin{multline*}
\int_{\R^n}\abs[big]{\langle \mat A\nabla^m w(x,t), \nabla^m \mathcal{E}\varphi(x,t)\rangle}\,dx
\\\leq
	C 
	\smash{\sum_{j=0}^{m-1}}\min\bigl(\doublebar{\nabla_\pureH^{j+1} \psi_j}_{L^{p'}(\R^n)}, t^{n/p'-n/r-1}\doublebar{\nabla_{\pureH}^j\psi_j}_{L^r(\R^n)}	\bigr)
	\\\times
	\sup_{\tau>0}\biggl(\int_{\R^n} \biggl(\fint_{B((x,\tau),\tau/2)} \abs{\nabla^m w}^2\biggr)^{p/2}\biggr)^{1/p}
.\end{multline*}
By assumption the term on the last line is finite, and so if $\varphi$ and thus $\psi_j$ is smooth and compactly supported, the bound \eqref{B:eqn:w:absolute} is valid. As an immediate corollary, $\langle \Tr_{m-1}^+\varphi,\M_{\mat A}^+ w\rangle_{\R^n}
=\langle \nabla^m \mathcal{E}\varphi,\mat A\nabla^m w\rangle_{\R^\dmn_+}$ for all $\varphi\in C^\infty_0(\R^\dmn)$.

We now turn to formula~\eqref{B:eqn:Neumann:regular}.
We seek to show that if $\varphi\in C^\infty_0(\R^\dmn)$, then
\begin{equation*}
\langle \nabla^m \varphi,\mat A\nabla^m w\rangle_{\R^\dmn_+}
=\langle \Tr_{m-1}^+\varphi,\M_{\mat A}^+ w\rangle_{\R^n}
=\langle \nabla^m \mathcal{E}\varphi,\mat A\nabla^m w\rangle_{\R^\dmn_+}
.\end{equation*}
Let $\eta_R(t)=\eta(t/R)$, where $\eta$ is smooth, supported in $B(0,2)$ and equal to $1$ in $B(0,1)$. An argument using the bound  \eqref{B:eqn:w:absolute} shows that as $R\to\infty$,
\begin{equation*}\langle \nabla^m (\eta_R \mathcal{E}\varphi),\mat A\nabla^m w\rangle_{\R^\dmn_+}\to \langle \nabla^m \mathcal{E}\varphi,\mat A\nabla^m w\rangle_{\R^\dmn_+}
=\langle \Tr_{m-1}^+\varphi,\M_{\mat A}^+ w\rangle_{\R^n}.
\end{equation*}
But by Lemma~\ref{B:lem:slices}, $\nabla^m w\in L^1_{loc}(\R^\dmn_+)$, and so if $\varphi$ is compactly supported, then by the weak formulation~\eqref{B:eqn:L} of $Lw=0$ we have that $\langle \nabla^m \varphi,\mat A\nabla^m w\rangle_{\R^\dmn_+}$ depends only on $\Tr_{m-1}^+\varphi$. Thus, if $\varphi$ is compactly supported then 
\begin{equation*}\langle \nabla^m \varphi,\mat A\nabla^m w\rangle_{\R^\dmn_+}=\langle \nabla^m (\eta_R \mathcal{E}\varphi),\mat A\nabla^m w\rangle_{\R^\dmn_+}\end{equation*}
for all $R$ large enough, and so formula~\eqref{B:eqn:Neumann:regular} is valid.

Finally, we come to the formula involving $\psi_j$.
Now, observe that
\begin{multline*}
\int_\varepsilon^T \langle \mat A\nabla^m u(\,\cdot\,,t), \nabla^m \mathcal{E}\varphi(\,\cdot\,,t)\rangle_{\R^n}\,dt
\\= 
	\sum_{j=0}^{m-1}\sum_{\substack{\abs\alpha=m \\ \abs\beta=m}}
	\int_\varepsilon^T\int_{\R^n}\partial^\alpha \biggl(\frac{t^{m-j-1}}{(m-j-1)!} \mathcal{Q}^m_t \psi_j(x)\biggr) \,\overline{A_{\alpha\beta}(x)\,\partial^\beta w(x,t)}\,dx\,dt
.\end{multline*}
We wish to bound terms on the right-hand side.

We begin with terms for which  $\alpha_\dmn>0$. Let $\alpha=\gamma+\vec e_\dmn$. We have that
\begin{multline*}
	\int_\varepsilon^T\int_{\R^n}\partial^\alpha \bigl({t^{m-j-1}} \mathcal{Q}^m_t \psi_j(x)\bigr) \,\overline{A_{\alpha\beta}(x)\,\partial^\beta w(x,t)}\,dx\,dt
\\\begin{aligned}
&= 
	\int_\varepsilon^T
	\int_{\R^n}\partial_t\partial^\gamma (t^{m-j-1} \mathcal{Q}^m_t \psi_j(x))\,\overline{A_{\gamma\beta}\,\partial^\beta w(x,t)}\,dx\,dt
.\end{aligned}\end{multline*}
Integrating by parts in $t$, we see that
\begin{multline*}
	\int_\varepsilon^T\int_{\R^n}\partial^\alpha \bigl({t^{m-j-1}} \mathcal{Q}^m_t \psi_j(x)\bigr) \,\overline{A_{\alpha\beta}(x)\,\partial^\beta w(x,t)}\,dx\,dt
\\\begin{aligned}
&= 
	-
	\int_\varepsilon^T
	\int_{\R^n}\partial^\gamma  (t^{m-j-1} \mathcal{Q}^m_t \psi_j(x))\,\overline{A_{\gamma\beta}\,\partial^\beta \partial_tw(x,t)}\,dx\,dt
	\\&\qquad
	+\int_{\R^n}\partial^\gamma  (T^{m-j-1} \mathcal{Q}^m_T \psi_j(x))\,\overline{A_{\gamma\beta}\,\partial^\beta w(x,T)}\,dx
	\\&\qquad
	-\int_{\R^n}\partial^\gamma (\varepsilon^{m-j-1} \mathcal{Q}^m_\varepsilon \psi_j(x))\,\overline{A_{\gamma\beta}\,\partial^\beta w(x,\varepsilon)}\,dx
.\end{aligned}\end{multline*}
By Lemma~\ref{B:lem:Neumann:slices}, the second and third terms have norm at most
\begin{multline*}
C\doublebar{\nabla_\pureH^j\psi_j}_{L^{p'}(\R^n)} \frac{1}{T}\doublebar{\mathcal{A}_2(t \1_{T/2<t<2T} \nabla^m w)}_{L^p(\R^n)}
\\+C\doublebar{\nabla_\pureH^j\psi_j}_{L^{p'}(\R^n)} \frac{1}{\varepsilon}\doublebar{\mathcal{A}_2(t \1_{\varepsilon/2<t<3\varepsilon/2} \nabla^m w)}_{L^p(\R^n)}
\end{multline*}
and thus satisfy our desired bounds.

We turn to the first term.
Applying the product rule, we have that
\begin{multline*}
\int_{\varepsilon}^T\int_{\R^n} \partial_\pureH^\gamma \partial_t^{m-\abs\gamma-1} (t^{m-j-1} \mathcal{Q}^m_t \psi_j(x))\,\overline{A_{\gamma\beta}(x)\,\partial^\beta \partial_t w(x,t)}\,dx\,dt
\\\begin{aligned}
&= 
	\sum_{\mathclap{\ell=\max(\abs\gamma,j)}}^{m-1}
	C_{m,j,\abs\gamma,\ell}
	\int_{\varepsilon}^T\int_{\R^n}
	 t^{\ell-j}
	\partial_t^{\ell-\abs\gamma}  \partial_\pureH^\gamma \mathcal{Q}^m_t \psi_j(x)
	\,\overline{A_{\gamma\beta}(x)\,\partial^\beta \partial_t w(x,t)}\,dx\,dt
.\end{aligned}\end{multline*}
We remark that if $\ell=j=\abs\gamma$, then $C_{m,j,\abs\gamma,\ell}=(m-j-1)!$. Recall that these terms are the terms that appear explicitly in the statement of this lemma, and so we will bound them later in this section.

If $\ell>j$ or $\ell>\gamma$, then by the bounds \eqref{B:eqn:area:2} and~\eqref{B:eqn:Q:area},
\begin{multline*}\int_0^\infty \int_{\R^n}
	 \abs[big]{t^{\ell-j}
	\partial_t^{\ell-\abs\gamma}  \partial_\pureH^\gamma \mathcal{Q}^m_t \psi_j(x)
	\,\overline{A_{\gamma\beta}(x)\,\partial^\beta \partial_t w(x,t)}}\,dx\,dt
\\\begin{aligned}
&\leq
C\int_{\R^n}\mathcal{A}_2(
	 t^{\ell-j}
	\partial_t^{\ell-\abs\gamma}  \partial_\pureH^\gamma \mathcal{Q}^m_t \psi_j)\,
	\mathcal{A}_2(t\,\partial^\beta \partial_t w)\,dx
\\&\leq
	C\doublebar{\nabla_\pureH^j \psi_j}_{L^{p'}(\R^n)}\doublebar{\mathcal{A}_2(t\,\nabla^m \partial_t w)}_{L^p(\R^n)}
\end{aligned}\end{multline*}
as desired.

We now consider the terms with $\alpha_\dmn=0$; we may write these terms as
\begin{multline*}
\sum_{\abs\beta=m}
	\sum_{\substack{\abs\alpha=m\\\alpha_\dmn=0}}
	\int_{\varepsilon}^T\int_{\R^n} \frac{t^{m-j-1}}{(m-j-1)!}\, \partial_\pureH^\alpha  \mathcal{Q}^m_t \psi_j(x)\,\overline{A_{\alpha\beta}(x)\,\partial^\beta w(x,t)}\,dx\,dt
\\=
\sum_{j=0}^{m-1} \int_\varepsilon^T
\frac{t^{m-j-1}}{(m-j-1)!}
\langle \nabla^m_{\pureH}\mathcal{Q}^m_t \psi_j(x),\overline{\mat A\nabla^m w(\,\cdot\,,t)}\rangle_{\R^n}\,dt.\end{multline*}
We again integrate by parts in~$t$ and see that
\begin{multline*}
\int_\varepsilon^T
\frac{t^{m-j-1}}{(m-j-1)!}
\langle \nabla^m_{\pureH}\mathcal{Q}^m_t \psi_j(x),\overline{\mat A\nabla^m w(\,\cdot\,,t)}\rangle_{\R^n}\,dt
\\\begin{aligned}
&=
	-\int_\varepsilon^T
	\frac{t^{m-j}}{(m-j)!}
	\frac{\partial}{\partial t}\langle \nabla^m_{\pureH}\mathcal{Q}^m_t \psi_j(x),\overline{\mat A\nabla^m w(\,\cdot\,,t)}\rangle_{\R^n}\,dt
	\\&\qquad 
	+\frac{T^{m-j}}{(m-j)!}
	\langle \nabla^m_{\pureH}\mathcal{Q}^m_T \psi_j(x),\overline{\mat A\nabla^m w(\,\cdot\,,T)}\rangle_{\R^n}
	\\&\qquad 
	-\frac{\varepsilon^{m-j}}{(m-j)!}
	\langle \nabla^m_{\pureH}\mathcal{Q}^m_\varepsilon \psi_j(x),\overline{\mat A\nabla^m w(\,\cdot\,,\varepsilon)}\rangle_{\R^n}
\end{aligned}\end{multline*}
We may bound the last two terms as before. If $0\leq j\leq m-1$,  then
\begin{multline*}
-\int_\varepsilon^T
	\frac{t^{m-j}}{(m-j)!}
	\frac{\partial}{\partial t}\langle \nabla^m_{\pureH}\mathcal{Q}^m_t \psi_j(x),\overline{\mat A\nabla^m w(\,\cdot\,,t)}\rangle_{\R^n}\,dt
\\\begin{aligned}
&=
	-\int_\varepsilon^T
	\frac{t^{m-j}}{(m-j)!}
	\langle \nabla^m_{\pureH}\partial_t\mathcal{Q}^m_t \psi_j(x),\overline{\mat A\nabla^m w(\,\cdot\,,t)}\rangle_{\R^n}\,dt
	\\&\qquad
	-\int_\varepsilon^T
	\frac{t^{m-j}}{(m-j)!}
	\langle \nabla^m_{\pureH}\mathcal{Q}^m_t \psi_j(x),\overline{\mat A\nabla^m \partial_t w(\,\cdot\,,t)}\rangle_{\R^n}\,dt
.\end{aligned}\end{multline*}
As before, the second term may be controlled by the bounds~\eqref{B:eqn:area:2} and~\eqref{B:eqn:Q:area}.
To control the first term, we integrate by parts in~$x$ and use  the fact that $Lw=0$. Then
\begin{multline*}
\langle \nabla^m_{\pureH}\partial_t\mathcal{Q}^m_t \psi_j(x),\overline{\mat A\nabla^m w(\,\cdot\,,t)}\rangle_{\R^n}
\\\begin{aligned}
&=
	\sum_{\substack{\abs\alpha=m\\\alpha_\dmn=0}}
	\sum_{\abs\beta=m}
	\langle \partial^\alpha\partial_t\mathcal{Q}^m_t \psi_j(x),\overline{ A_{\alpha\beta}\partial^\beta w(\,\cdot\,,t)}\rangle_{\R^n}
\\&=
	\sum_{\substack{\abs\gamma\leq m-1\\\gamma_\dmn=0}}
	\sum_{\abs\beta=m}(-1)^{m+\abs\gamma+1}
	\langle \partial_\pureH^\gamma\partial_t\mathcal{Q}^m_t \psi_j(x),\overline{ A_{\alpha\beta}\partial^\beta\partial_t^{m-\abs\gamma} w(\,\cdot\,,t)}\rangle_{\R^n}
.\end{aligned}\end{multline*}
Thus, by formula~\eqref{B:eqn:area:2}, 
{\multlinegap=0pt\begin{multline*}
\abs[bigg]{\int_\varepsilon^T
	\frac{t^{m-j}}{(m-j)!}
	\langle \nabla^m_{\pureH}\partial_t\mathcal{Q}^m_t \psi_j(x),\overline{\mat A\nabla^m w(\,\cdot\,,t)}\rangle_{\R^n}\,dt}
\\\leq
C\sum_{\substack{\abs\gamma\leq m-1\\\gamma_\dmn=0}} \doublebar{\mathcal{A}_2^{1/2}(t^{\abs\gamma-j+1} \partial_\pureH^\gamma \partial_t \mathcal{Q}^{m}_t  \psi_j)}_{L^{p'}(\R^n)} \doublebar{\mathcal{A}_2^{1/2}(t^{m-\abs\gamma}\nabla^m\partial_t^{m-\abs\gamma}w)}_{L^{p'}(\R^n)}.\end{multline*}}%
By the Caccioppoli inequality, the second term is at most $C\doublebar{\mathcal{A}_2(t\,\nabla^m \partial_t w)}_{L^{p'}(\R^n)}$. 
By the bound~\eqref{B:eqn:Q:area}, the first term is at most $C\doublebar{\nabla_\pureH^j \psi_j}_{L^{p'}(\R^n)}$, as desired.

Assembling our estimates, we see that
{\multlinegap=0pt\begin{multline*}
\int_\varepsilon^T \langle \mat A\nabla^m u(\,\cdot\,,t), \nabla^m \mathcal{E}\varphi(\,\cdot\,,t)\rangle_{\R^n}\,dt
\\\begin{aligned}
&= 
	-\sum_{j=0}^{m-1}\sum_{\abs\beta=m}
	\sum_{\substack{\abs\gamma=j\\\gamma_\dmn=0}}
	\frac{1}{(m-j-1)!}
	\int_{\R^n}\int_0^\infty 
	\partial_\pureH^\gamma \mathcal{Q}^{m}_t  \psi_j(x)
	\,\overline{A_{\gamma\beta}(x)\,\partial^\beta \partial_t w(x,t)}\,dt\,dx
	\\&\qquad+OK_{\varepsilon,T}(w,\varphi)
\end{aligned}\end{multline*}}%
as desired.
\end{proof}

To complete the proof of Theorems~\ref{B:thm:Neumann:1} and~\ref{B:thm:Neumann:2}, we must bound terms of the form
\begin{equation*}
	\int_\varepsilon^T
	\sum_{\abs\beta=m}
	\sum_{\substack{\abs\gamma=j\\\gamma_\dmn=0}}
	\int_{\R^n}
	\partial_\pureH^\gamma \mathcal{Q}^{m}_t  \psi_j(x)
	\,\overline{A_{\gamma\beta}(x)\,\partial^\beta u(x,t)}\,dx\,dt.\end{equation*}
for $0\leq j\leq m$, where $u=v$ or $u=\partial_t w$.

Choose some $j$ with $0\leq j\leq m$.
As usual, we integrate by parts in~$t$. If $\ell\geq 0$ is an integer, then
\begin{multline*}
\int_\varepsilon^T  \int_{\R^n}
	\partial_\pureH^\gamma \mathcal{Q}^{m}_t  \psi_j(x)
	\,\overline{A_{\gamma\beta}(x)\,\partial^\beta \partial_t^\ell u(x,t)}\,t^\ell\,dx\,dt
\\\begin{aligned}
&=
	-\frac{1}{\ell+1}\int_\varepsilon^T  \int_{\R^n}
	\partial_\pureH^\gamma \mathcal{Q}^{m}_t  \psi_j(x)
	\,\overline{A_{\gamma\beta}(x)\,\partial^\beta\partial_t^{\ell+1} u(x,t)}\,t^{\ell+1}\,dx\,dt
	\\&\qquad
	-\frac{1}{\ell+1}\int_\varepsilon^T  \int_{\R^n}
	\partial_\pureH^\gamma \partial_t\mathcal{Q}^{m}_t  \psi_j(x)
	\,\overline{A_{\gamma\beta}(x)\,\partial^\beta\partial_t^{\ell} u(x,t)}\,t^{\ell+1}\,dx\,dt
	\\&\qquad 
	+
	\frac{1}{\ell+1}\int_{\R^n}
	\partial_\pureH^\gamma \mathcal{Q}^{m}_T  \psi_j(x)
	\,\overline{A_{\gamma\beta}(x)\,\partial^\beta\partial_T^{\ell} u(x,T)}\,T^{\ell+1}\,dx
	\\&\qquad 
	-
	\frac{1}{\ell+1}\int_{\R^n}
	\partial_\pureH^\gamma \mathcal{Q}^{m}_\varepsilon  \psi_j(x)
	\,\overline{A_{\gamma\beta}(x)\, \partial^\beta\partial_\varepsilon^{\ell} u(x,\varepsilon)}\,\varepsilon^{\ell+1}\,dx
.\end{aligned}\end{multline*}
The second integral may be controlled by the bounds~\eqref{B:eqn:area:2} and~\eqref{B:eqn:Q:area} as usual. By Lemma~\ref{B:lem:Neumann:slices}, the last integral has norm at most
\begin{equation*}C\doublebar{\nabla_\pureH^j\psi_j}_{L^{p'}(\R^n)} \doublebar{\mathcal{A}_2(\1_{\varepsilon/2<t<3\varepsilon/2}t\nabla^m u)}_{L^p(\R^n)}\end{equation*}
so is uniformly bounded and approaches zero as $\varepsilon\to 0$. Similarly, the third integral is uniformly bounded and approaches zero as $T\to \infty$.

By induction,
\begin{multline*}
\int_{\R^n}\int_0^\infty 
	\partial_\pureH^\gamma \mathcal{Q}^{m}_t \psi_{j}(x)\,\overline{A_{\gamma\beta}(x)\,\partial^\beta u(x,t)}\,dt\,dx
\\\begin{aligned}
&=
	\frac{1}{(2k)!}\int_{\R^n}\int_0^\infty 
	\partial_\pureH^\gamma \mathcal{Q}^{m}_t \psi_{j}(x)\,\overline{A_{\gamma\beta}(x)\,\partial^\beta \partial_t^{2k} u(x,t)}\,t^{2k}\,dt\,dx
	+OK
\end{aligned}\end{multline*}
for any integer $k\geq 0$, where the term $OK$ is uniformly bounded and approaches a limit as $\varepsilon\to 0^+$ and $T\to\infty$. We have that
\begin{multline*}
	\smash{\sum_{\abs\beta=m}
	\sum_{\substack{\abs\gamma=j\\\gamma_\dmn=0}}} \int_{\R^n}\int_0^\infty 
	\partial_\pureH^\gamma \mathcal{Q}^{m}_t \psi_{j}(x)\,\overline{A_{\gamma\beta}(x)\,\partial^\beta \partial_t^{2k} u(x,t)}\,t^{2k}\,dt\,dx
\\\begin{aligned}
&=
	\int_0^\infty
	\langle \nabla^m \partial_t^{2k} u(\,\cdot\,,t),
	\mat A^*_{mj}\nabla_\pureH^j \mathcal{Q}^{m}_{t} \psi_{j}\rangle_{\R^n}
	t^{2k}\,dt
\end{aligned}\end{multline*}
where $\mat A^*_{mj}$ is the matrix that satisfies
\begin{equation*}(\mat A^*_{mj} \arr \psi)_\beta=\sum_{\substack{\abs\gamma=j\\\gamma_\dmn=0}} A^*_{\beta\gamma} \psi_\gamma \qquad\text{for any $\abs\beta=m$}.\end{equation*}
By Lemma~\ref{B:lem:control:lusin}, if $k$ is large enough then
\begin{multline*}
\int_0^\infty
	t^{2k} \abs{\langle \mat A^*_{mj}\nabla_\pureH^j \mathcal{Q}^{m}_{t} \psi_{j} ,\nabla^{m}\partial_t^{2k} u(\,\cdot\,,t)\rangle_{\R^n}}\,dt
\\\leq
C_{k}
\int_{4/3}^{4}
	\int_{\R^n}
	\mathcal{A}_2^-(\abs{t}^{k-2m+1}
		\partial_\dmn^{k-m} \s^{L^*}_\nabla(\mat A^*_{mj}\nabla_\pureH^j \mathcal{Q}^{m}_{\abs{t}r} \psi_{j}) )(x)
	\,\mathcal{A}_2^+(t\,{\nabla^m u})(x)
	\,dx
	\,dr 
.\end{multline*}
Define
\begin{equation*}R_t^r \arr\psi(z)
=
t^{k-2m+1}{ \partial_\dmn^{k-m}\s^{L^*}_\nabla (\mat A^*_{mj}\mathcal{Q}^m_{tr} \arr \psi)(z,-t) }.
\end{equation*}
Observe that $\mathcal{P}_t=\mathcal{Q}_{tr}$ is also an approximate identity with a Schwartz kernel.
By the bound \eqref{B:eqn:S:Schwartz:p}, for any fixed $r$ with $2/3<r<8$ and any $p'$ with $1<p'<\infty$ we have 
$L^{p'}$ boundedness of $\psi\mapsto\mathcal{A}_2(R_t^r\arr\psi)$. 
Thus,
\begin{multline*}
\abs[bigg]{\int_0^\infty
	t^{2k}\langle \mat A^*_{mj}\nabla_\pureH^j \mathcal{Q}^{m}_{t} \psi_{j},\nabla^{m}\partial_\dmn^{2k} u(\,\cdot\,,t)\rangle_{\R^n}\,dt}
\\\leq 
	C\doublebar{\nabla_\pureH^j\psi_j}_{L^{p'}(\R^n)}
	\doublebar{\mathcal{A}_2(t\,\nabla^m u)}_{L^p(\R^n)}
\end{multline*}
as desired.


\newcommand{\etalchar}[1]{$^{#1}$}
\providecommand{\bysame}{\leavevmode\hbox to3em{\hrulefill}\thinspace}
\providecommand{\MR}{\relax\ifhmode\unskip\space\fi MR }
\providecommand{\MRhref}[2]{%
  \href{http://www.ams.org/mathscinet-getitem?mr=#1}{#2}
}
\providecommand{\href}[2]{#2}

\end{document}